\documentclass{elsarticle}
\usepackage{geometry,amsmath,mathrsfs,amsthm,amssymb,cases,mathrsfs,mathtools,txfonts}
\usepackage{hyperref,soul,color,xcolor}

\newtheorem{theorem}{Theorem}[section]
\newtheorem{lemma}[theorem]{Lemma}
\newtheorem{proposition}[theorem]{Proposition}
\newtheorem{corollary}[theorem]{Corollary}
\theoremstyle{definition}

\theoremstyle{remark}
\newtheorem{remark}[theorem]{Remark}
\numberwithin{equation}{section}

\begin{document}
\title{Entire Solutions and Asymptotic Behavior to a Class of Parabolic $k$-Hessian Equations with More General Right-Hand Side Terms}
\author[BNU]{Ning An}
\ead{AnNingAN@mail.bnu.edu.cn}

\author[BNU]{Jiguang Bao\corref{mycorrespondingauthor}}
\cortext[mycorrespondingauthor]{Corresponding author}
\ead{jgbao@bnu.edu.cn}

\address[BNU]{School of Mathematical Sciences, Beijing Normal University,\\ Laboratory of Mathematics and Complex Systems, Ministry of Education,
Beijing 100875, China}

\begin{abstract}
This paper investigates entire classical separable variable radial solutions to a class of parabolic $k$-Hessian equations $-u_t \left(\sigma_k(\lambda(D^2u))\right)^{\alpha}=f(|x|)g(t)$, where the right-hand side consists of positive continuous functions. Without assuming that $-u_t$ is positively bounded, we extend previous work for the equations which right-hand side is $1$ to more general cases, including cases where the right-hand side terms are bounded and periodic. By employing the method of Euler’s broken line, we obtain the existence and nonexistence results. Furthermore, we obtain the precise asymptotic power exponent.
\end{abstract}

\maketitle

\section{Introduction and main results}
In this paper, we consider the existence of entire solutions to a class of parabolic $k$-Hessian equations
\begin{equation}
-u_t \left(\sigma_k(\lambda(D^2u))\right)^{\alpha}=f(|x|)g(t),\quad\text{in }\mathbb{R}^{n+1}_-\coloneqq\mathbb{R}^n\times(-\infty,0].\label{one1}
\end{equation}
In \eqref{one1}, $\sigma_k(\lambda)$ denotes the $k$-th elementary symmetric function for $k=1,\cdots,n$, which means
\begin{equation}
\sigma_k(\lambda)=\sum_{1\leq i_1<\cdots<i_k\leq n}\lambda_{i_1}\cdots\lambda_{i_k},\notag
\end{equation}
$\lambda(D^2u)$ denotes the eigenvalues of the Hessian matrix $D^2u$, and $\alpha\in\mathbb{R}\backslash\{0\}$. The right-hand side (RHS) of \eqref{one1} is of separable variable form and satisfies the following hypothesis:
\[f\text{ and }g\text{ are given positive continuous functions, }f\text{ has a finite upper bound on }[0,+\infty),\tag{H1}\]
and we denote this upper bound by $M_f$.
\par
The fully nonlinear operator $\sigma_k(\lambda(D^2u))$ is a generalization of Laplace ($k=1$) operator and Monge-Amp\`ere operator ($k=n$), which both have significant applications in theoretical and practical fields. The $k$-Hessian equations themselves also play an important role in researching special Lagrangian geometry \cite{CNS}, the Minkowski type problem \cite{CW} and prescribed Weingarten curvature problem \cite{GLM}.
\par
For $\alpha=1$, there are many studies focussing on the Liouville type theorem of the solutions to
\begin{equation}-u_t\sigma_k(\lambda(D^2u))=1.\label{pkH 1}\end{equation}
For $k=n$, with a necessary assumption that the partial derivative of the solution with respect to time is bounded in $\mathbb{R}^{n+1}_-$, \cite{GH} extends the classical J${\rm \ddot{o}}$rgens-Calabi-Pogorelov theorem for elliptic Monge-Amp\`ere equations to the parabolic case.
\begin{theorem}[\cite{GH}]\label{GH 1}
Let \( u \in C^{4,2} (\mathbb{R}^{n+1}_- )\) be a parabolically convex solution to \eqref{pkH 1} with $k=n\geq2$. Suppose that there exist constants \( m_1 \geq m_2 > 0 \) such that for all \( (x, t) \in \mathbb{R}^n \times (-\infty, 0] \), \( -m_1 \leq u_t (x, t) \leq -m_2\). Then \( u \) has the form
\begin{equation}\label{u=1+1}
u(x, t) = -Ct + p(x),\text{ where } C > 0 \text{ is a constant and } p(x) \text{ is a convex quadratic polynomial}.
\end{equation}
\end{theorem}
\cite{ZBW} reduces the regularity requirement of $u$ in Theorem \ref{GH 1} to $C^{2,1}$. For general $k$, \cite{BCGJ} obtains that any strictly convex solution to $\sigma_k(\lambda(D^2u))=1$ with a lower quadratic growth condition must be a quadratic polynomial. On this basis, \cite{NT} develops the method for elliptic case given in \cite{BCGJ}, and applies it to the parabolic case. \cite{NT} proves that any parabolically convex solution $u$ to \eqref{pkH 1} with assumptions that $u_t$ is bounded and $u(x,0)\geq A|x|^2-B$ for constants $A$, $B>0$, must also be of the form \eqref{u=1+1}. The optimal expression form of Liouville type theorem for \eqref{pkH 1} is still receiving continuous attention. When the assumptions are reduced or removed, the solutions may not be as \eqref{u=1+1}.
\par
In fact, there are examples occurring in the above situation. The viscosity solutions to \eqref{pkH 1} may not of the form \eqref{u=1+1}, but are radially symmetric in $x$-variable and of variable separated form
\begin{equation}
u(x,t)=w(t)v(r),\label{one2}
\end{equation}
where $r=\left| x\right|\in[0,\infty)$, $t\in(-\infty,0]$. \cite{GH} (for $k=n$) and \cite{NT} (for general $k$) give the following particular viscosity solution:
\[u(x,t)=(-(k+1)t)^{\frac{1}{k+1}}\cdot\frac{k+1}{(2k)^{\frac{k}{k+1}}\left(C_n^k(1+k-\frac{2k}{n})\right)^{\frac{1}{k+1}}}|x|^{\frac{2k}{k+1}},\]
which is of the form \eqref{one2}. Inspired by this example, \cite{ABL} first considers the entire solution to \eqref{pkH 1} for $k=n$ with $u_t$ is unbounded, and obtains the existence of the smooth solutions of the form \eqref{one2} and their asymptotic behavior. As a corollary, \cite{ABL} also investigates
the Liouville type theorem for parabolic Monge-Amp\`ere equations in the sense of the separable variable radial solution.
\par
For different $\alpha\in\mathbb{R}\backslash\{0\}$, there are also significant research findings when considering about
\begin{equation}\label{alpha RHS=1}
-u_t \left(\sigma_k(\lambda(D^2u))\right)^{\alpha}=1.
\end{equation}
For $k=n$, more general parabolic Monge-Amp\`ere equations, including \eqref{alpha RHS=1} with $\alpha=\pm 1/n$ and several other meaningful forms, are studied in \cite{XB}, while Theorem \ref{GH 1} is also generalized. \cite{LT} considers the case that $\alpha=1/(n+2)$ and investigates the equations as affine normal flows. Recently, \cite{CJ} generalizes the results given in \cite{ABL} from parabolic Monge-Amp\`ere equations to \eqref{alpha RHS=1}. The authors of \cite{CJ} also conduct the existence, uniqueness and nonexistence of smooth entire classical solutions to \eqref{alpha RHS=1} with $u_t$ is unbounded. Also as corollaries, they obtain the asymptotic behavior of those solutions and a Liouville type theorem for \eqref{alpha RHS=1} in the sense of the separable variable radial solution.
\par
When the RHS of \eqref{one1} is more general than $1$, like being bounded or periodic, \eqref{one1} has also been studied in several papers. For bounded RHS, especially for the case that the RHS is a perturbation of $1$ at infinity, we refer to \cite{QB} (for elliptic Monge-Amp\`ere case and the RHS is H${\rm \ddot{o}}$lder continuous), \cite{ZGB} (for parabolic Monge-Amp\`ere case), \cite{BLZ} (for parabolic Monge-Amp\`ere case with new asymptotic behavior), \cite{WB} (for elliptic $k$-Hessian case), \cite{ZhouB} (for parabolic $k$-Hessian case) and \cite{AB} (for parabolic $k$-Hessian case with new asymptotic behavior). For periodic RHS, \cite{CY} proves that the solutions to elliptic Monge-Amp\`ere equations must be as a quadratic polynomial plus a periodic function of $x$. Then, the case that the RHS function is asymptotically close to a periodic function is considered by \cite{TZ}, and the authors prove that the function given by the solution minus a specific quadratic polynomial must be also asymptotically close to a periodic function. And \cite{ZB} generalizes this work to parabolic Monge-Amp\`ere equations $-u_t\det D^2u=f(x)$, where $f$ is a sufficiently smooth positive periodic function of $x$. Therefore, the equations which we focus on in this paper all have rich research backgrounds and theoretical significance.
\par
In this paper, we promote the method used in \cite{ABL} and \cite{CJ}, consider \eqref{one1} with more general right-hand side terms and obtain the existence and nonexistence. In fact, when $1$ is the RHS, we can set $f\equiv 1\equiv g$ and then the RHS satisfies (H1), so our work is a natural promotion of previous work. Hereinafter, we always assume that the RHS of \eqref{one1} satisfies (H1).
\begin{theorem}[existence and nonexistence]\label{the1}
When $\alpha>0$, the equation \eqref{one1} has a positive classical parabolically $k-$convex solution $u(x,t)$ of the form \eqref{one2}, which means there exists a decreasing function $w(t)$ and a $k-$convex function $V(x)=v(r)$ such that $u(x,t)=w(t)v(r)>0$, is the classical solution to \eqref{one1} on $\mathbb{R}^{n+1}_-$. When $-1/k<\alpha<0$, the equation \eqref{one1} admits no positive classical parabolically $k$-convex entire solution of the form \eqref{one2}.
\end{theorem}
\begin{remark}
The cases where $\alpha<-1/k$ is omitted here for ensuring the existence of $w(t)$. In fact, this case may deduce \eqref{two4} is meaningless in the real sense. When $\alpha=-1/k$, the entire solution stated in Theorem \ref{the1} still exists. This case is quite similar to Remark 4.1 in \cite{CJ} and we omit it here. As for the case where $\alpha=0$, the equation \eqref{one1} degenerates.
\end{remark}
Then, we add the hypothesis that $f$ and $g$ are both upper and lower bounded on their domains of definition, which means there exists constants $m_f$, $M_f$, $m_g$ and $M_g$ such that
\[f\text{ and }g\text{ satisfy (H1), }0<m_f\leq f\text{ on }[0,\infty)\text{ and }0<m_g\leq g\leq M_g<\infty\text{ on }(-\infty,0].\tag{H2}\]
Under the hypothesis (H2), we can obtain a rough asymptotic behavior of the entire solution to \eqref{one1} given in Theorem \ref{the1}.
\begin{proposition}\label{pro1}
Suppose $f$ and $g$ satisfy {\rm (H2)} and 
\(\alpha>\max\{0, \frac{2k-n}{kn}\}\), then the solution $u(x,t)=w(t)v(r)$ given by Theorem \ref{the1} satisfies that $u(x,t)\to\infty$ if and only if $-t+|x|^2\to\infty$. And there exist two constants $\Lambda_1$ and $\Lambda_2$ depending on $n$, $k$, $\alpha$ and $f$, such that
\[\Lambda_1r^{\frac{2k\alpha}{k\alpha+1}}\leq v(r)\leq 1+\Lambda_2 r^{\frac{2k\alpha}{k\alpha+1}},\quad r\geq0,\]
where we assume that $v(0)=1$.
\end{proposition}
If the positive continuous functions $f$ and $g$ further satisfy the hypothesis:
\[f\text{ and }g\text{ are periodic, with }T_f\text{ and }T_g\text{ are the minimal positive periods respectively,}\tag{H3}\]
then we have
\begin{equation}\label{f and g period}
f(r+T_f)=f(r),\quad \forall r\geq0,\quad\text{and}\quad g(t-T_g)=g(t),\quad\forall t\leq0,
\end{equation}
and naturelly $f$ and $g$ satisfy (H2). We will obtain the following refined asymptotic behavior theorem.
\begin{theorem}[Asymptotic behavior]\label{the2}
Suppose $f$ and $g$ satisfy {\rm (H3)} and 
\(\alpha>\max\{0, \frac{2k-n}{kn}\}\), then the $u(x,t)=w(t)v(r)$ given by Theorem \ref{the1} has the following refined asymptotic behavior at infinity:
\begin{equation}
w(t)=C_w(-t)^{\frac{1}{k\alpha+1}}+O\left((-t)^{\frac{1}{k\alpha+1}-1}\right),\quad v(r)=C_vr^{\frac{2k\alpha}{k\alpha+1}}+o(r^{\frac{2k\alpha}{k\alpha+1}}),\label{one6}
\end{equation}
where the constant $C_w$ is only dependent on $k$, $\alpha$ and $g$, and $C_v$ is only dependent on $n$, $k$, $\alpha$ and $f$.
\end{theorem}
\begin{remark}
The asymptotic behavior of $u$ given in Theorem \ref{the2} is the generalization of Theorem 1.3 in \cite{ABL} and Theorem 1.2 in \cite{CJ}. To ensure the existence of the solution given by Theorem \ref{the1}, we only consider the case $\alpha>0$, so the cases $k>n/2$ is what really matters in the research of asymptotic behavior. Owing to space constraints, we have not carried out the computation of the asymptotic behavior of $v$ in the critical case where $k>n/2$ and $\alpha=(2k-n)/(kn)$. Nevertheless, we believe that similar results to  Theorem 1.3 in \cite{CJ} can be derived by following our approach given in this paper. When $k>n/2$, in the cases where $0<\alpha<(2k-n)/(kn)$, the precise asymptotic behavior of $v$ remains unclear, but one can obtain a rough bound estimate by L’Hospital’s rule. For this, we refer to Remark 5.2 in \cite{CJ}.
\end{remark}
\begin{corollary}[Refined asymptotic behavior when $\alpha=1$]\label{the3}
Suppose $f$ and $g$ satisfy \eqref{f and g period} and $\alpha=1$, then the asymptotic behavior of $v$ given in \eqref{one6} can be refined as
\[v(r)=Cr^{\frac{2k}{k+1}}+r^{-\frac{2}{k+1}}F_1(r)+r^{-\frac{k+3}{k+1}}F_2(r)+O(r^{-M_{n,k}}),\]
where $C$ is simply the value of $C_v$ when $\alpha=1$, $F_1$ and $F_2$ are $T_f$-periodic functions with the integral over a period is zero, and $M_{n,k}$ is a constant only depending on $n$ and $k$ with a clear formulation given in Section 4.
\end{corollary}
\begin{remark}
The refined asymptotic behavior of $u$ when $\alpha=1$ given in Corollary \ref{the3} is the generalization of Theorem 1.4 in \cite{ABL} and Theorem 1.4 in \cite{CJ}. The periodic forcing term on RHS of \eqref{one1} poses a challenge to the original method, requiring more refined estimates to balance the effects introduced by the periodic term.
\end{remark}
The paper is organized as follows. In Section 2, we use Euler's broken line method and classical theories of ODE to obtain the existence and nonexistence of the solution to \eqref{one1} with the form \eqref{one2}, and accomplish the proof of Theorem \ref{the1}. In Section 3, the asymptotic behavior of the solution constructed in Theorem \ref{the1} will be given after meticulous analysis and calculation. We also finish the proof of Proposition \ref{pro1} and Theorem \ref{the2} in this section. When $\alpha=1$, we obtain refined asymptotic behavior by isolating the leading term of $v$ near infinity at Section 4. The core technical challenge is to achieve a complete isolation. We overcome it through constructing correction periodic functions whose integral average is zero and give a precise convergence order of the remainder term.

\section{Proof of Theorem \ref{the1}}
\subsection{Preliminaries}
Consider about the existence of the solution to \eqref{one1} with the form \eqref{one2}. It is easy to see that
\begin{equation}
	\frac{\partial u}{\partial x_i}(x)=w(t)v'(r)\frac{x_i}{r},\quad i=1, \dots, n, \notag
\end{equation}
\begin{equation}
	\frac{\partial ^2 u}{\partial x_i \partial x_j}(x)=w(t)\left((\frac{v''(r)}{r^2}-\frac{v'(r)}{r^3})x_ix_j+\frac{v'(r)}{r}\delta_{ij}\right), \quad i, j=1, \dots, n, \notag
\end{equation}
and therefore, with $a\coloneqq w(t)(rv''(r)-v'(r))/{r^3}$ and $b\coloneqq w(t)v'(r)/r$,
\begin{equation}
\notag D^2u=ax^Tx+bI.
\end{equation}
By knowledge of linear algebra, we have
\begin{equation}
	\lambda(D^2u)=(ar^2+b,b,\dots,b)=\left(w(t)v''(r), w(t)\frac{v'(r)}{r}, \dots, w(t)\frac{v'(r)}{r}\right). \notag
\end{equation}
Then for any $r>0$,
\begin{equation}\label{two1}
	\sigma_k(\lambda(D^2u))= w^k(t) \left({C_{n-1}^{k-1}}v''(r)\left(\frac{v'(r)}{r}\right)^{k-1}+{C_{n-1}^{k}}\left(\frac{v'(r)}{r}\right)^k\right).
\end{equation}
\par
Substituting \eqref{two1} and $u_t=w'(t) v(r)$ into \eqref{one1}, we have
\begin{equation}\notag
	-w'(t)w^{k\alpha}(t)v(r)\left({C_{n-1}^{k-1}}v''(r)\left(\frac{v'(r)}{r}\right)^{k-1}+{C_{n-1}^{k}}\left(\frac{v'(r)}{r}\right)^k\right)^{\alpha}=f(r)g(t). 
\end{equation}
The above equation can be converted into the following equations
\begin{numcases}\,	-w'(t) w^{k\alpha}(t)=Cg(t),\quad t<0, \notag\\
v(r)\left({C_{n-1}^{k-1}}v''(r)\left(\frac{v'(r)}{r}\right)^{k-1}+{C_{n-1}^{k}}\left(\frac{v'(r)}{r}\right)^k\right)^{\alpha}=\frac{f(r)}{C},\quad r>0,\notag
\end{numcases}
where $C$ is an arbitrary positive constant.
\begin{remark}
We can only consider the case that $C=1$ without loss of generality. For the proof, we refer to \cite{ABL}.
So hereinafter, we only need to consider the equations
\begin{numcases}\,	-w'(t) w^{k\alpha}(t)=g(t),\quad t<0,\label{two2}\\
v(r)\left({C_{n-1}^{k-1}}v''(r)\left(\frac{v'(r)}{r}\right)^{k-1}+{C_{n-1}^{k}}\left(\frac{v'(r)}{r}\right)^k\right)^{\alpha}=f(r),\quad r>0.\label{two3}
\end{numcases}
\end{remark}
It is obvious that for any $\alpha\in(-\frac{1}{k},0)\cup(0,+\infty)$, the solution to \eqref{two2} must be as 
\begin{equation}
	w(t)=\left(w^{k\alpha+1}(0)+(k\alpha+1)\int_t^0g(s)\,{\rm d}s\right)^{\frac{1}{k\alpha+1}}.  \notag
\end{equation}
In fact, we need $\alpha>-1/k$ to make sure that $w(t)$, the solution to \eqref{two2}, is well-defined, and there is a further explanation in Remark 1.1 of \cite{CJ}.\par
To prove the existence of solutions to \eqref{two3}, we focus on the existence of the solutions to the initial value problem:
\begin{equation}\label{v fangcheng} 
\left\{ \begin{aligned}
&v'(r)=\left(\frac{nr^{k-n}}{C_n^k}\int_0^rs^{n-1}\left(\frac{f(s)}{v(s)}\right)^{\frac{1}{\alpha}}\,{\rm d}s\right)^{\frac{1}{k}},\quad r>0, \\
&v(0)=1. \end{aligned}\right.
\end{equation}
\begin{lemma}\label{v C2}
For any $\alpha\in\mathbb{R}\backslash\{0\}$, if $v$ is a solution to \eqref{v fangcheng}, then $v\in C^2([0,\infty))$ and $v$ is also the classical solution to \eqref{two3} on $[0,\infty)$.
\end{lemma}
\begin{proof}[Proof]
For any $r>0$, we can know from the equation that $v$, as the solution to \eqref{v fangcheng}, is $C^2$ at $r$. We will focus on the regularity of $v$ near $r=0$ next. It can be observed that
\begin{equation}
0\leq \frac{nr^{k-n}}{C_n^k}\int_0^rs^{n-1}\left(\frac{f(s)}{v(s)}\right)^{\frac{1}{\alpha}}\,{\rm d}s\leq \frac{nr^{k-1}}{C_n^k}\int_0^r\left(\frac{f(s)}{v(s)}\right)^{\frac{1}{\alpha}}\,{\rm d}s,\quad r>0,\notag
\end{equation}
and
\begin{equation}
r^{k-1}\int_0^r\left(\frac{f(s)}{v(s)}\right)^{\frac{1}{\alpha}}\,{\rm d}s\to 0,\quad\text{as }r\to 0.\notag
\end{equation}
Thus, we have
\begin{equation}
\lim_{r\to 0}v'(r)=\lim_{r\to 0}\left(\frac{nr^{k-n}}{C_n^k}\int_0^rs^{n-1}\left(\frac{f(s)}{v(s)}\right)^{\frac{1}{\alpha}}\,{\rm d}s\right)^{\frac{1}{k}}=0,\notag
\end{equation}
which means $v\in C^1([0,\infty))$ with $v'(0)=0$. Then by L'Hospital's rule, we have
\begin{equation}
\lim\limits_{r\to0}\left(\frac{v'(r)-v'(0)}{r-0}\right)^k=\lim_{r\to 0}\frac{\frac{n}{C_n^k}\int_0^rs^{n-1}\left(\frac{f(s)}{v(s)}\right)^{\frac{1}{\alpha}}\,{\rm d}s}{r^n}=\lim\limits_{r\to0}\frac{1}{C_n^k}\left(\frac{f(r)}{v(r)}\right)^{\frac{1}{\alpha}}=\frac{(f(0))^{\frac{1}{\alpha}}}{C_n^k}.\notag
\end{equation}
therefore $v\in C^2([0,\infty))$ with $v''(0)=\left((f(0))^{\frac{1}{\alpha}}/C_n^k\right)^{\frac{1}{k}}>0$.\par
By \eqref{v fangcheng}, we deduce that
\[r^{n-k}(v'(r))^k=\frac{n}{C_n^k}\int_0^rs^{n-1}\left(\frac{f(s)}{v(s)}\right)^{\frac{1}{\alpha}}\,{\rm d}s.\]
After taking the derivatives of both sides of the above equation, we have
\[\left(r^{n-k}(v'(r))^k\right)'=\frac{nr^{n-1}}{C_n^k}\left(\frac{f(r)}{v(r)}\right)^{\frac{1}{\alpha}},\quad r>0,\]
and therefore, for any $r>0$,
\[k(v'(r))^{k-1}v''(r)+\frac{n-k}{r}(v'(r))^{k}=\frac{nr^{k-1}}{C_n^k}\left(\frac{f(r)}{v(r)}\right)^{\frac{1}{\alpha}}.\]
It is obvious that $v$ satisfies \eqref{two3} after a direct computation from the above equation.
\par
In conclusion, we have proved that the solution to the given initial value problem \eqref{v fangcheng} is in $C^2([0,\infty))$ and is also the classical solution to \eqref{two3}.
\end{proof}

\begin{remark}\label{rem2}
Without loss of generality, we assume $w(0)=(k\alpha+1)^{1/(k\alpha+1)}$ hereinafter and obtain that
\begin{equation}\label{two4}
w(t)=(k\alpha+1)^{\frac{1}{k\alpha+1}}\left(1+\int_t^0g(s)\,{\rm d}s\right)^{\frac{1}{k\alpha+1}},\quad t\leq0.
\end{equation}
And we only need to consider the case that $v(0)=1$ without loss of generality. The proof is quite similar to that given in Remark 2.1 and Remark 2.2 in \cite{ABL}.
\end{remark}

\subsection{Proof of local existence when $\alpha\neq0$}
In this subsection, we construct a Euler's broken line to proof that there exist local solutions to \eqref{v fangcheng}. Note that local existence holds for any $\alpha\neq0$, including the case $\alpha\leq-1/k$.
\begin{lemma}[Local existence]\label{v local}
For any $\alpha\in\mathbb{R}\backslash\{0\}$, there is a local solution to the initial value problem \eqref{v fangcheng}. More explicitly, there exists a constant $R>0$ and a $v_{\text{loc}}(r)\in C^2([0,R))$ solving \eqref{v fangcheng} in $[0,R)$. 
\end{lemma}
\begin{proof}[Proof]
	Define
	\begin{equation}
		\mathscr{R} \coloneqq \{ v \in C^0([0, R)): v \text{ is increasing on }(0,R), v(0)=1, v(R)<1+h\},\notag
	\end{equation}
and a map from $[0,R)\times\mathscr{R}$ to $\mathbb{R}$\begin{equation}\notag
		F(r,v)\coloneqq\left(\frac{nr^{k-n}}{C_n^k}\int_0^rs^{n-1}\left(\frac{f(s)}{v(s)}\right)^{\frac{1}{\alpha}}\,{\rm d}s\right)^{\frac{1}{k}}
	\end{equation}
where $h$ is any given positive constant and $R$ is a positive number to be determined. Since $v\geq v(0)=1$, $F(r,v)>0$ for $r>0$ and $F(0,v)=0$. And $F(r, \cdot)$ is increasing for $\alpha<0$ and decreasing for $\alpha>0$.\par
Define an Euler's broken line $\hat{v}$ on $[0,R)$ as
	\begin{equation}
		\begin{cases}\hat{v}(0)=1 \notag \\ \hat{v}(r)=\hat{v}(r_{i-1})+F(r_{i-1},\hat{v})(r-r_{i-1}),\quad r_{i-1}\leq r< r_i, \end{cases}
	\end{equation}
where $0=r_0<r_1< \cdots <r_m=R$ with $m\in\mathbb{N}_+$ and $r_1,\cdots,r_m$ to be determined. In fact, for any $r_{i-1}\leq r<r_i$, the definition of $\hat{v}(r)$ relies only on the value of $F(r_{i-1},\hat{v})$, thus relies only on the value of $\hat{v}$ on $(0,r_{i-1})$, so the definition makes sense. And we claim that for given $h>0$, there exists a enough small $R>0$ such that $\hat{v}\in\mathscr{R}$ and therefore $\hat{v}$ is well-defined. 
\par
If $\alpha>0$, for any $r\in(0, R)$ and any $v\in\mathscr{R}$, we have
\[0<F(r,v)\leq(C_n^k)^{-\frac{1}{k}}M_f^{\frac{1}{k\alpha}}r.\]
Thus if we set 
\[R<(C_n^k)^{\frac{1}{2k}}M_f^{-\frac{1}{2k\alpha}}h^{\frac{1}{2}},\] 
it can be easily seen that $\hat{v}\in\mathscr{R}$ after successive calculatuion from $i=1$ to $i=m$.\par
If $\alpha<0$, we use the fact that $f$ is positive, continuous and upper-bounded to find a $m(R)$ such that
\[0<m(R)\leq f(r)\leq M_f,\quad 0\leq r<R,\]
and therefore for any $r\in(0, R)$ and any $v\in\mathscr{R}$, 
\begin{equation}
0<F(r,v)\leq(C_n^k)^{-\frac{1}{k}}\left(\frac{m(R)}{v(R)}\right)^{\frac{1}{k\alpha}}r. \notag
\end{equation}
In fact, we can combine the above estimate with that for the case $\alpha>0$, namely, we have
\begin{equation}\label{F estimate}
0<F(r,v)\leq(C_n^k)^{-\frac{1}{k}}M_f^{\frac{1+sgn(\alpha)}{2k\alpha}}\left(\frac{m(R)}{v(R)}\right)^{\frac{1-sgn(\alpha)}{2k\alpha}}r,\quad 0<r<R,
\end{equation}
where $sgn(\cdot)$ is the sign function. We assume that $h\leq(C_n^k)^{-\frac{1}{k}}(m(R)/\hat{v}(R))^{\frac{1}{k\alpha}}R^2$, and then can obtain that
\[\frac{m(R)}{\hat{v}(R)}\leq(C_n^k)^{\alpha}\left(\frac{h}{R^2}\right)^{k\alpha}.\]
Take the limit as $R\to0^+$ on both sides of the above equation,  and obtain
\[0<f(0)=\lim\limits_{R\to0^+}\frac{m(R)}{\hat{v}(R)}\leq(C_n^k)^{\alpha}\lim\limits_{R\to0^+}\left(\frac{h}{R^2}\right)^{k\alpha}=0,\]
which makes a contradiction. Thus for sufficiently small $R>0$, $(C_n^k)^{-\frac{1}{k}}(m(R)/\hat{v}(R))^{\frac{1}{k\alpha}}R^2<h$. And it can be easily seen that $\hat{v}\in\mathscr{R}$ by using \eqref{F estimate} and $\hat{v}$ is well-defined. Hereinafter, we fix the value of $R$ to make sure the claim is true.
\par
Then by \eqref{F estimate} and for any $v\in\mathscr{R}$ and any $0<r<R$, we further obtain that
\[0<F(r,v)\leq(C_n^k)^{-\frac{1}{k}}M_f^{\frac{1+sgn(\alpha)}{2k\alpha}}\left(\frac{m(R)}{1+h}\right)^{\frac{1-sgn(\alpha)}{2k\alpha}}r,\]
and after taking the limit as $r\to0$, have
\[\lim\limits_{r\to0}F(r,v)=\lim\limits_{r\to0}(C_n^k)^{-\frac{1}{k}}M_f^{\frac{1+sgn(\alpha)}{2k\alpha}}\left(\frac{m(R)}{1+h}\right)^{\frac{1-sgn(\alpha)}{2k\alpha}}r=0,\]
which convergence is  uniform for $v\in\mathscr{R}$. Hence for any positive number $\epsilon$, there is a $\bar{r}\in (0,R)$ such that for any $r: 0\leq r \leq \bar{r}$, $F(r,v)<\epsilon/2$. Next, we will choose $r_1,\cdots,r_m$ to construct $\hat{v}$ as an $\epsilon-$approximation solution of \eqref{v fangcheng}.
\par
Firstly, for any $r\in[0,\bar{r}]$, using the definition of $\hat{v}(r)$, we have
	\begin{equation}
		\frac{\,{\rm d}\hat{v}(r)}{\,{\rm d}r}=F(r_{i-1},\hat{v}),\quad r_{i-1}< r < r_i \notag
	\end{equation}
and we denote $\,{\rm d}\hat{v}(r_{i-1})/\,{\rm d}r=\hat{v}'_+(r_{i-1})$. It implies for any $r\in[0,\bar{r}]$,
	\begin{equation}
		\left|\frac{\,{\rm d}\hat{v}(r)}{\,{\rm d}r}-F(r,\hat{v})\right|=\left| F(r_{i-1},\hat{v})-F(r,\hat{v})\right|\leq\left| F(r_{i-1},\hat{v})\right|+\left| F(r,\hat{v})\right|. \notag
	\end{equation}
Because of $r_{i-1}\leq r\leq\bar{r}$, we have $F(r_{i-1},\hat{v})<\epsilon/2$ and $F(r,\hat{v})<\epsilon/2$. Hence,
	\begin{equation}
		\left|\frac{\,{\rm d}\hat{v}(r)}{\,{\rm d}r}-F(r,\hat{v})\right|<\frac{\epsilon}{2}+\frac{\epsilon}{2}=\epsilon. \notag
	\end{equation}
	\par
Secondly, for any $r: \bar{r}\leq r<R$, assume that $r_{j-1}\leq r<r_j$,
\begin{align}
		\left|\frac{\,{\rm d}\hat{v}(r)}{\,{\rm d}r}-F(r,\hat{v})\right|&=\left| F(r_{j-1},\hat{v})-F(r,\hat{v})\right|\notag \\
		&=\left(\frac{n}{C_n^k}\right)^{\frac{1}{k}}\left|\left(r_{j-1}^{k-n}\int_0^{r_{j-1}}s^{n-1}\left(\frac{f(s)}{\hat{v}(s)}\right)^{\frac{1}{\alpha}}\,{\rm d}s\right)^{\frac{1}{k}}- \left(r^{k-n}\int_0^rs^{n-1}\left(\frac{f(s)}{\hat{v}(s)}\right)^{\frac{1}{\alpha}}\,{\rm d}s\right)^{\frac{1}{k}}\right|. \notag 
	\end{align}
Noticing that $|x-y|\leq|x^k-y^k|^{1/k}$ for all $x>0$ and $y>0$, we have 
	\begin{align}
		\left|\frac{\,{\rm d}\hat{v}(r)}{\,{\rm d}r}-F(r,\hat{v})\right|&\leq\left(\frac{n}{C_n^k}\right)^{\frac{1}{k}}\left|r_{j-1}^{k-n}\int_0^{r_{j-1}}s^{n-1}\left(\frac{f(s)}{\hat{v}(s)}\right)^{\frac{1}{\alpha}}\,{\rm d}s-r^{k-n}\int_0^rs^{n-1}\left(\frac{f(s)}{\hat{v}(s)}\right)^{\frac{1}{\alpha}}\,{\rm d}s\right|^{\frac{1}{k}} \notag \\
		&\leq\left(\frac{n}{C_n^k}\right)^{\frac{1}{k}}\left(\left|r_{j-1}^{k-n}-r^{k-n}\right|\int_0^{r_{j-1}}s^{n-1}\left(\frac{f(s)}{\hat{v}(s)}\right)^{\frac{1}{\alpha}}\,{\rm d}s+ r^{k-n}\int_{r_{j-1}}^rs^{n-1}\left(\frac{f(s)}{\hat{v}(s)}\right)^{\frac{1}{\alpha}}\,{\rm d}s\right)^{\frac{1}{k}}.\label{vhat cha1}
	\end{align}
Set $C_1\coloneqq M_f^{\frac{1+sgn(\alpha)}{2\alpha}}\left(m(R)/(1+h)\right)^{\frac{1-sgn(\alpha)}{2\alpha}}=C(\alpha,f,h,R)$, then we have
\begin{equation}\label{vhat cha2}
\left|r_{j-1}^{k-n}-r^{k-n}\right|\int_0^{r_{j-1}}s^{n-1}\left(\frac{f(s)}{\hat{v}(s)}\right)^{\frac{1}{\alpha}}\,{\rm d}s\leq C_1\left|r_{j-1}^{k-n}-r^{k-n}\right|\int_0^{r_{j-1}}s^{n-1}\,{\rm d}s\leq\frac{C_1R^n}{n}\left|r_{j-1}^{k-n}-r^{k-n}\right|,
\end{equation}
and
\begin{equation}\label{vhat cha3}
 r^{k-n}\int_{r_{j-1}}^rs^{n-1}\left(\frac{f(s)}{\hat{v}(s)}\right)^{\frac{1}{\alpha}}\,{\rm d}s\leq C_1\bar{r}^{k-n}\int_{r_{j-1}}^r s^{n-1}\,{\rm d}s=\frac{C_1\bar{r}^{k-n}}{n}\left|r^n-r_{j-1}^n\right|
\end{equation}
with that $r^{k-n}$ and $r^n$ are both Liptchitz continuous, for $\epsilon$ determined above, there exists $\delta(\epsilon)>0$ such that for any $r_1$, $r_2$ with $|r_1-r_2|<\delta(\epsilon)$,
\begin{equation}\notag
\left|r_1^{k-n}-r_2^{k-n}\right|<\frac{C_n^k}{2C_1R^n}\epsilon^k,\quad\left|r_1^{n}-r_2^{n}\right|<\frac{C_n^k}{2C_1\bar{r}^{k-n}}\epsilon^k.
\end{equation}
Using the above estimates, \eqref{vhat cha1} , \eqref{vhat cha2} and \eqref{vhat cha3},  once $|r-r_{i-1}|<\delta(\epsilon)$, we still have
\begin{equation}
\left|\frac{\,{\rm d}\hat{v}(r)}{\,{\rm d}r}-F(r,\hat{v})\right|<\frac{\epsilon}{2}+\frac{\epsilon}{2}=\epsilon.\notag
\end{equation}
\par
For this, assume $r_1=\bar{r}$ and $\max \limits_ {2\leq i \leq m}\left| r_{i-1}-r_i\right|<min\{\delta(\epsilon),\bar{r}\}$, then $[0, R)$ is well divided and $\hat{v}_{\epsilon}(r)$ can also be well constructed. For any $r\in[0, R)$,
\begin{equation}
\left|\frac{\,{\rm d}\hat{v}_{\epsilon}(r)}{\,{\rm d}r}-F(r,\hat{v})\right|<\epsilon.\notag
\end{equation}
Thus we prove that the constructed Euler's broken line $\hat{v}_{\epsilon}(r)$ is an $\epsilon$-approximation solution to \eqref{v fangcheng} on $[0, R)$.
	\par
Take a positive constant sequence $\{\epsilon_j\}^{\infty}_{j=1}: \lim\limits_{j \to \infty}\epsilon_j=0$. Once we imitate the above process, $\{\hat{v}_j\}^{\infty}_{j=1}$ can be constructed satisfying for any $j \in \mathbb{N}_+$ and any $r\in[0, R)$,
	\begin{equation}\label{three8}
		\left|\frac{\,{\rm d}\hat{v}_j(r)}{\,{\rm d}r}-F(r,\hat{v}_j)\right|<\epsilon_j.
	\end{equation}
\par
By \eqref{F estimate}, we obtain for any $r\in(0,R)$ and $v\in\mathscr{R}$,
\[0<F(r,v)\leq (C_n^k)^{-\frac{1}{k}}M_f^{\frac{1+sgn(\alpha)}{2k\alpha}}\left(\frac{m(R)}{1+h}\right)^{\frac{1-sgn(\alpha)}{2k\alpha}}R=:C_2,
\]
where $C_2=C(n,k,\alpha,f,h,R)$. It can be kowned by differential mean value theorem that for any $r_1, r_2\in[0,R)$ and any $j\in\mathbb{N}^{+}$, 
\begin{equation}
\left|\hat{v}_j(r_1)-\hat{v}_j(r_2)\right|\leq C_2\left|r_1-r_2\right|. \notag
	\end{equation}
Therefore, for any $\epsilon>0$, there is a $\delta=\epsilon/C_2$ such that $\left|\hat{v}_j(r_1)-\hat{v}_j(r_2)\right|<\epsilon$ for any $r_1, r_2\in[0, R): \left| r_1-r_2\right|<\delta$ and any $j\in\mathbb{N}^{+}$, which means $\{\hat{v}_j\}^{\infty}_{j=1}$ is equicontinuous. For any $r\in[0, R)$, because of 
	\begin{equation}
		\left|\frac{\,{\rm d}\hat{v}_j(r)}{\,{\rm d}r}\right|-\left| F(r,\hat{v}_j)\right|\leq\left|\frac{\,{\rm d}\hat{v}_j(r)}{\,{\rm d}r}-F(r,\hat{v}_j)\right|<\epsilon_j,  \notag
	\end{equation}
we have
	\begin{equation}
		\left|\frac{\,{\rm d}\hat{v}_j(r)}{\,{\rm d}r}\right|<\epsilon_j+\left| F(r,\hat{v}_j)\right|\leq\epsilon_j+C_2.  \notag
	\end{equation}
Hence,
	\begin{equation}
		\left|\hat{v}_j(r)\right|\leq 1+(\epsilon_j+C_2)r\leq 1+(\epsilon_j+C_2)R. \notag
	\end{equation}
By $\lim\limits_{j \to \infty}\epsilon_j=0$, we know $\{\epsilon_j\}^{\infty}_{j=1}$ has an upper bound which is denoted as $\bar{\epsilon}$. Therefore, for any $j\in\mathbb{N}^{+}$ and any $r\in[0, R)$, 
\begin{equation}
\left|\hat{v}_j(r)\right|\leq 1+(\bar{\epsilon}+C_2)R,\notag
\end{equation}
which means $\{\hat{v}_j\}^{\infty}_{j=1}$ is uniformly bounded. According to Arzela-Ascoli Theorem, $\{\hat{v}_j\}_{j=1}^\infty$ has a uniformly convergent subsequence, might as well still denoted as $\{\hat{v}_j\}_{j=1}^\infty$. 
\par
Let $v_{\text{loc}}\coloneqq\lim\limits_{j \to \infty}\hat{v}_j$,  then $v_{\text{loc}}(0)=1$, $v_{\text{loc}}'(0)=0$. By \eqref{three8}, we have for $r\in[0,R)$,
	\begin{equation}
		\frac{\,{\rm d}\hat{v}_j(r)}{\,{\rm d}r}=F(r,\hat{v}_j)+\Delta_j(r),\quad \left|\Delta_j(r)\right|<\epsilon_j.  \notag
	\end{equation}
Calculating the integral and using the initial value, we have
	\begin{equation}
		\hat{v}_j(r)=1+\int_0^r F(s,\hat{v}_j(s))\,{\rm d}s+\int_0^r \Delta_j(s)\,{\rm d}s,\quad \left|\Delta_j(r)\right|<\epsilon_j.  \notag
	\end{equation}
Let $j \to \infty$, we have
	\begin{equation}\label{three9}
		v_{\text{loc}}(r)=1+\int_0^r F(s,v_{\text{loc}}(s))\,{\rm d}s.
	\end{equation}
By the uniformity of $\lim\limits_{j\to\infty}\hat{v}_j=v_{\text{loc}}$ and the continuity of each $\hat{v}_j$, we know $v_{\text{loc}}$ is continuous and differentiable.
\par
Substituting \eqref{three9} into \eqref{v fangcheng}, we can verify that $v_{\text{loc}}(r)\in C^2([0,R))$ is the local solution on $[0, R)$.
\end{proof}

\subsection{Proof of entire existence and uniqueness when $\alpha>0$}
In this subsection, we always assume that $\alpha>0$ and denote $\alpha'=1/\alpha$.
For the existence of the initial value problem \eqref{v fangcheng}, it remains to prove that the local solution given by Lemma \ref{v local} can be extended to $[0,\infty)$. Besides, we prove the uniqueness of the solution to \eqref{v fangcheng}.
\begin{lemma}[Entire existence]\label{v entire}
There exists an entire solution $v(r)$ to the initial value problem \eqref{v fangcheng}.
\end{lemma}
\begin{proof}
The core part of the proof is illustrating that the local solution to \eqref{v fangcheng} is locally bounded on $[0,\infty)$. Let $R>0$ and $v_{\text{loc}}\in C^2([0,R))$ are as given in Lemma \ref{v local}. By \eqref{v fangcheng}, we have $v_{\text{loc}}'\geq0$ and then for any $r\in[0,R)$, $v_{\text{loc}}(r)\geq v_{\text{loc}}(0)=1$. Therefore, the following estimate can be obtained:
\begin{equation}
(v_{\text{loc}}'(r))^k=\frac{nr^{k-n}}{C_n^k}\int_0^rs^{n-1}\left(\frac{f(s)}{v_{\text{loc}}(s)}\right)^{\alpha'}\,{\rm d}s\leq\frac{M_f^{\alpha'}}{C_n^k}r^{k},\notag
\end{equation} 
and then,
\[v_{\text{loc}}'(r)\leq\left(\frac{M_f^{\alpha'}}{C_n^k}\right)^{\frac{1}{k}}r=:2C_3r,\]
where $C_3=C(n,k,\alpha,f)$ is independent with $r$ or $R$. For any $r\in[0,R)$, integrating from $0$ to $r$ and using $v(0)=1$, we have
\begin{equation}\label{v fangcheng1}
1\leq v_{\text{loc}}(r)\leq1+C_3r^2<1+C_3R^2,\quad 0\leq r<R,
\end{equation}
which means the solution to \eqref{v fangcheng} is locally bounded.\par
Eventually, the local solution to \eqref{v fangcheng} on $[0,R)$ can be extended to $[0,\infty)$ according to the knowledge of ordinary differential equations, otherwise there will be a contradiction with \eqref{v fangcheng1}.
\end{proof}
\begin{lemma}[Uniqueness]\label{v unique}
The solution to \eqref{v fangcheng} given by Lemma \ref{v entire} is unique on $C^2([0,\infty))$.
\end{lemma}
\begin{proof}
For any given $R>0$, define $F(r,v)$ as in the proof of Lemma \ref{v local} and consider about the following problem:
\begin{equation} \label{v fangcheng2}
\left\{ \begin{aligned}v'(r)&=F(r,v),\quad 0<r<R,\\
v(0)&=1.\end{aligned} \right.
\end{equation}
Suppose it has two solutions on $[0,R)$, denoted by $v$ and $\tilde{v}$. For any $r\in[0,R)$, define a map from $[0,1]$ to $\mathbb{R}$ as
\begin{equation}
h(\lambda)\coloneqq\left(\frac{nr^{k-n}}{C_n^k} \int_0^r s^{n-1}\left(\frac{f(s)}{v(s)+\lambda(\tilde{v}(s)-v(s))}\right)^{\alpha'}\,{\rm d}s\right)^{\frac{1}{k}},\quad\lambda\in[0,1],\notag
\end{equation}
then we have $h(0)=F(r,v)$, $h(1)=F(r,\tilde{v})$, and
\begin{equation}
h'(\lambda)=-\frac{\alpha'}{k}\left(\frac{nr^{k-n}}{C_n^k}\right)^{\frac{1}{k}}
\left(\int_0^r \frac{s^{n-1}(f(s))^{\alpha'}}{(v(s)+\lambda(\tilde{v}(s)-v(s)))^{\alpha'}}\,{\rm d}s\right)^{\frac{1}{k}-1}
\int_0^r\frac{s^{n-1}(f(s))^{\alpha'}(\tilde{v}(s)-v(s))}{(v(s)+\lambda(\tilde{v}(s)-v(s)))^{\alpha'+1}}\,{\rm d}s.\notag
\end{equation}
Applying the Newton-Leibnitz formula of $h(\lambda)$ on $[0,1]$, we have
\begin{align}
|F(r,\tilde{v})-F(r,v)|&=|h(1)-h(0)|=\left|\int_0^1h'(\lambda)\,{\rm d}\lambda\right|\notag\\
&\leq\frac{\alpha'}{k}\left(\frac{nr^{k-n}}{C_n^k}\right)^{\frac{1}{k}}\int_0^1  \left(\left(\int_0^r \frac{s^{n-1}f^{\alpha'}}{(v+\lambda(\tilde{v}-v))^{\alpha'}}\,{\rm d}s\right)^{\frac{1}{k}-1}
\int_0^r\frac{s^{n-1}f^{\alpha'}(\tilde{v}-v)}{(v+\lambda(\tilde{v}-v))^{\alpha'+1}}\,{\rm d}s\right)\,{\rm d}\lambda.\label{v fangcheng3}
\end{align}
Samely with the deducing of \eqref{v fangcheng1}, we have $1\leq v(r)\leq1+C_3R^2$ and $1\leq\tilde{v}(r)\leq1+C_3R^2$. Therefore, for any $\lambda\in[0,1]$, we have
\begin{equation}
1\leq v(r)+\lambda(\tilde{v}(r)-v(r))\leq1+C_3R^2.\notag
\end{equation}
Then combining with \eqref{v fangcheng3}, we have
\begin{align}
|F(r,\tilde{v})-F(r,v)|&\leq\frac{\alpha'}{k}\left(\frac{nr^{k-n}}{C_n^k}\right)^{\frac{1}{k}}
\int_0^1  \left(\left(\int_0^r \frac{s^{n-1}f^{\alpha'}}{(1+C_3R^2)^{\alpha'}}\,{\rm d}s\right)^{\frac{1}{k}-1}
\int_0^rs^{n-1}f^{\alpha'}(\tilde{v}-v)\,{\rm d}s\right)\,{\rm d}\lambda \notag\\
&\leq\frac{\alpha'M_f}{k}\left(\frac{n}{C_n^k}\right)^{\frac{1}{k}}\left(\frac{1+C_3R^2}{m(R)}\right)^{\alpha'(1-\frac{1}{k})}r^{\frac{k-n}{k}} ||\tilde{v}-v||_{C^0([0,r])}\int_0^1  \left( \int_0^rs^{n-1}\,{\rm d}s\right)^{\frac{1}{k}-1}\left(\int_0^r s^{n-1}\,{\rm d}s\right)\,{\rm d}\lambda\notag\\
&=\leq\frac{\alpha'M_f}{k(C_n^k)^{\frac{1}{k}}}\left(\frac{1+C_3R^2}{m(R)}\right)^{\alpha'(1-\frac{1}{k})}r||\tilde{v}-v||_{C^0([0,r])}=:C_4r||\tilde{v}-v||_{C^0([0,r])},\label{v fangcheng4}
\end{align}
where $C_4=C(n,k,\alpha,f,R)$ is independent with $r$.
\par
Integrating \eqref{v fangcheng2} for both $v$ and $\tilde{v}$, and using the initial value, we have
\[v(r)=1+\int_0^rF(s,v)\,{\rm d}s,\quad \tilde{v}(r)=1+\int_0^rF(s,\tilde{v})\,{\rm d}s,\]
and then,
\begin{equation}\label{v fangcheng5}
|\tilde{v}(r)-v(r)|=|(\tilde{v}(r)-1)-(v(r)-1)|\leq\int_0^r|F(s,\tilde{v})-F(s,v)|\,{\rm d}s.
\end{equation} 
Using the estimate \eqref{v fangcheng4} and \eqref{v fangcheng5}, we have
\begin{equation}
|\tilde{v}(r)-v(r)|\leq C_4\int_0^r s||\tilde{v}-v||_{C^0([0,s])}\,{\rm d}s,\notag
\end{equation}
then,
\begin{equation}
||\tilde{v}-v||_{C^0([0,r])}=\max\limits_{0\leq \tau \leq r}|\tilde{v}(\tau)-v(\tau)|\leq C_4\max\limits_{0\leq \tau \leq r} \int_0^{\tau} s||\tilde{v}-v||_{C^0([0,s])}\,{\rm d}s = C_4 \int_0^r s||\tilde{v}-v||_{C^0([0,s])}\,{\rm d}s.\label{v fangcheng6}
\end{equation}
\par
For any $r\in[0,R)$, define 
\[I(r)\coloneqq \int_0^rs||\tilde{v}-v||_{C^0([0,s])}\,{\rm d}s,\]
then $I(0)=0$, $I(r)\geq 0$ and $I$ is monotone increasing. By \eqref{v fangcheng6}, we have
\begin{equation}
I'(r)=r||\tilde{v}-v||_{C^0([0,r])}\leq C_4r\cdot I(r)\leq C_4R\cdot I(r).\notag
\end{equation}
Then, it is easy to verify that
\begin{equation}
\left({\rm e}^{-C_1Rr}I(r)\right)'={\rm e}^{-C_1Rr}(I'(r)-C_1R\cdot I(r))\leq 0,\notag
\end{equation}
which implies for any $r\in[0,R)$,
\begin{equation}
0\leq {\rm e}^{-C_1Rr}I(r) \leq I(0)=0,\notag
\end{equation}
thus $I(r)\equiv 0$. By the definition of $I$, we have $||\tilde{v}-v||_{C^0([0,R))}=0$, and then $v\equiv\tilde{v}$ on $[0,R)$, which means the entire solution to \eqref{v fangcheng} is unique on $[0,R)$. Therefore, the uniqueness is trivial with the arbitrariness of $R>0$ and Lemma \ref{v C2}.
\end{proof}

\subsection{Proof of Theorem \ref{the1}}
Recalling \eqref{two4}, Lemma \ref{v C2}, Lemma \ref{v entire} and Lemma \ref{v unique}, we know that the parabolically $k-$convexity and the regularity of $u$ are the only work left to be completed for the proof of Theorem \ref{the1}.
\begin{proof}[Proof of existence]
We set $\alpha>0$. By \eqref{two4}, we can compute that for any $t\leq 0$,
\begin{equation}\label{pkc t}
w'(t)=-(k\alpha+1)^{\frac{1}{k\alpha+1}-1}\left(1+\int_t^0g(s)\,{\rm d}s\right)^{\frac{1}{k\alpha+1}-1}g(t)<0.\end{equation}
For $r=0$, we know from \eqref{two1} and the proof of Lemma \ref{v C2} that
\begin{equation}\sigma_j(\lambda(D^2u))=C_n^jw^j(t)\left(\frac{(f(0))^{\frac{1}{\alpha}}}{C_n^k}\right)^{\frac{j}{k}}>0,\quad 1\leq j\leq k.\label{pkc r1}\end{equation}
For any $r>0$, we define a positive function of $r$ as
 \[J(r)\coloneqq\frac{v'(r)}{r^{1-\frac{n}{k}}},\]
and then by \eqref{v fangcheng}, we have
\[(J(r))^k=\frac{n}{C_n^k}\int_0^r s^{n-1}\left(\frac{f(s)}{v(s)}\right)^{\frac{1}{\alpha}}\,{\rm d}s.\]
Differentiating the above equation, we obtain that
\[J'(r)=\frac{n}{kC_n^k}\frac{r^{n-1}}{(J(r))^{k-1}}\left(\frac{f(r)}{v(r)}\right)^{\frac{1}{\alpha}}>0.\]
Together with a direct computation about $J'(r)$, it is clear that \(v''(r)>(1-n/k)v'(r)/r\). Then for any $1\leq j\leq k$, \eqref{one1} implies
\begin{align}\notag
	\sigma_j(\lambda(D^2u))&>w^j(t) \left({C_{n-1}^{j-1}}(1-\frac{n}{k})\left(\frac{v'(r)}{r}\right)^j+{C_{n-1}^{j}}\left(\frac{v'(r)}{r}\right)^j\right)\\ &=C_{n-1}^{j-1}w^j(t)\left(\frac{v'(r)}{r}\right)^j(\frac{k-n}{k}+\frac{n-j}{j})\notag\\
&\geq 0.\label{pkc r2}
\end{align}
Summarizing \eqref{pkc t}, \eqref{pkc r1} and \eqref{pkc r2}, we obtain the parabolically $k-$convexity of the solution we constructed.
\par
By \eqref{two4}, we know that $w(t)$ is $C^1((-\infty,0])$. With this and Lemma \ref{v C2}, we prove the regularity of $u$ and know that it is the classical solution of \eqref{one1}. Then, we finish the proof of the existence part of Theorem \ref{the1}.
\end{proof}
If we want to improve the regularity of $u(x,t)$, the regularity of $f$ and $g$ should also be promoted by assumption. For any $\gamma\geq1$, by \ref{two4}, we know 
\begin{equation}\label{w' biaodashi}
u_t(x,t)=w'(t)v(|x|)=\frac{-g(t)}{(k\alpha+1)\left(1+\int_t^0g(s)\,{\rm d}s\right)}u(x,t),
\end{equation}
and once $g\in C^\gamma((-\infty,0])$, $w\in C^{\gamma+1}((-\infty,0])$. By \eqref{w' biaodashi}, for any given $t<0$, $u(\cdot,t)$ is a classical solution to
\[\left(\sigma_k^{\frac{1}{k}}(\lambda(D^2u))\right)(\cdot,t)=\left(\frac{fg}{-u_t}\right)^{\frac{1}{k\alpha}}(\cdot,t)=\left((1+k\alpha)(1+\int_t^0g(s)\,{\rm d}s)\right)^{\frac{1}{k\alpha}}\left(\frac{f(|\cdot|)}{u(\cdot,t)}\right)^{\frac{1}{k\alpha}}=:\mathscr{F}(\cdot,t).\]
Therefore by Evans-Krylov theory and relevant regularity theory for $k-$Hessian equation that $u(\cdot,t)\in C^{\gamma}(B_1)$, once $f\in C^\gamma(0,\infty)$, $\mathscr{F}(\cdot,t)\in C^\gamma(B_1)$ can be deduced with $u>1$, and then so is $u(\cdot,t)$. When $r\geq1$, we can also have a similar induction result by \eqref{v fangcheng}.
\begin{proof}[Proof of nonexistence]
We set $-1/k<\alpha<0$. By \eqref{v fangcheng}, we know for any $r>0$,
\[v'(r)=\left(\frac{nr^{k-n}}{C_n^k}\int_0^rs^{n-1}\left(\frac{f(s)}{v(s)}\right)^{\frac{1}{\alpha}}\,{\rm d}s\right)^{\frac{1}{k}}=\left(\frac{n}{C_n^k}\right)^{\frac{1}{k}}r^{-(\frac{n}{k}-1)}\left(\int_0^r(f(s))^{\alpha'}s^{n-1}(v(s))^{-\alpha'}\,{\rm d}s\right)^{\frac{1}{k}}.\]
And using (H2), we have
\[M_f^{\alpha'}\int_0^r s^{n-1}(v(s))^{-\alpha'}\,{\rm d}s\leq\int_0^r(f(s))^{\alpha'}s^{n-1}(v(s))^{-\alpha'}\,{\rm d}s\leq m_f^{\alpha'}\int_0^r s^{n-1}(v(s))^{-\alpha'}\,{\rm d}s,\]
and therefore,
\begin{equation}\label{jiabi nonexistence}\left(\frac{nM_f^{\alpha'}}{C_n^k}\right)^{\frac{1}{k}}r^{-(\frac{n}{k}-1)}\left(\int_0^r s^{n-1}(v(s))^{-\alpha'}\,{\rm d}s\right)^{\frac{1}{k}}\leq v'(r)\leq\left(\frac{nm_f^{\alpha'}}{C_n^k}\right)^{\frac{1}{k}}r^{-(\frac{n}{k}-1)}\left(\int_0^r s^{n-1}(v(s))^{-\alpha'}\,{\rm d}s\right)^{\frac{1}{k}}.\end{equation}
\par
When $-1/k<\alpha<0$, $\frac{1}{k+1}(\frac{1}{\alpha}-1)<-1$, and then for an arbitrary $c>0$, we have
\[\int_c^{+\infty}\left(\int_0^t s^{-\frac{1}{\alpha}}\,{\rm d}s\right)^{-\frac{1}{k+1}}\,{\rm d}t=\left(1-\frac{1}{\alpha}\right)^{\frac{1}{k+1}}\int_c^{+\infty}t^{\frac{1}{k+1}(\frac{1}{\alpha}-1)} \,{\rm d}t<+\infty.\]
We apply Lemma 2.1 of \cite{CJ} on the following two initial value problems of unknown function $v$:
\begin{equation}\notag
\left\{ \begin{aligned}
&v'(r)=\left(\frac{nm_f^{\alpha'}}{C_n^k}\right)^{\frac{1}{k}}r^{-(\frac{n}{k}-1)}\left(\int_0^r s^{n-1}(v(s))^{-\alpha'}\,{\rm d}s\right)^{\frac{1}{k}},\quad r>0, \\
&v(0)=1, \end{aligned}\right.
\end{equation}
and
\begin{equation}\notag
\left\{ \begin{aligned}
&v'(r)=\left(\frac{nM_f^{\alpha'}}{C_n^k}\right)^{\frac{1}{k}}r^{-(\frac{n}{k}-1)}\left(\int_0^r s^{n-1}(v(s))^{-\alpha'}\,{\rm d}s\right)^{\frac{1}{k}},\quad r>0, \\
&v(0)=1.\end{aligned}\right.
\end{equation}
There is no solution $v$ to both of the above problem satisfying $v\in C^1([0,+\infty))\cap C^2(0,+\infty)$ with $v(0)=1$, $v'(0)=0$, $v'(r)>0$ and $v''(r)>0$ for any $r>0$. By \eqref{jiabi nonexistence} and comparison theorem for first-order ODEs, we can prove the nonexistence of the entire solution of the form \eqref{one2} to \eqref{one1}.
\end{proof}

\section{The asymptotic behavior of $u$ given in Theorem \ref{the1}}

In this section, we will obtain the asymptotic behavior of $u(x,t)$ given in Theorem \ref{the1} at infinity, prove Proposition \ref{pro1} and Theorem \ref{the2}. We also set $\alpha>0$ and denote $\alpha'=1/\alpha$ in this section. And (H2) gives the upper and lower boundness of $f$ and $g$ on $\mathbb{R}$.
\begin{proof}[Proof of Proposition \ref{pro1}]
By \eqref{v fangcheng} and the increasing proposition of $v$, we have the following estimate:
\[(v'(r))^k\geq\frac{nr^{k-n}}{C_n^k}\int_0^r s^{n-1}\left(\frac{f(s)}{v(r)}\right)^{\alpha'}\,{\rm d}s\geq\frac{r^{k-n}}{C_n^k}\left(\frac{m_f}{v(r)}\right)^{\alpha'}\int_0^rns^{n-1}\,{\rm d}s=\frac{m_f^{\alpha'}}{C_n^k(v(r))^{\alpha'}}r^k.\]
The above estimate implies that
\begin{equation}\label{v esti 1}
(v(r))^{\frac{\alpha'}{k}}v'(r)\geq\left(\frac{m_f^{\alpha'}}{C_n^k}\right)^{\frac{1}{k}}r.
\end{equation}
Noticing that $\left(v^{\frac{k+\alpha'}{k}}(r)\right)'=\frac{k+\alpha'}{k}(v(r))^{\frac{\alpha'}{k}}v'(r)$, we can deduce from \eqref{v esti 1} that
\[\left(v^{\frac{k+\alpha'}{k}}(r)\right)'\geq\frac{k+\alpha'}{k}\left(\frac{m_f^{\alpha'}}{C_n^k}\right)^{\frac{1}{k}}r.\]
After integrating the above from $0$ to $r$, we have for any $r\geq0$,
\[v^{\frac{k+\alpha'}{k}}(r)\geq1+\frac{k+\alpha'}{2k}\left(\frac{m_f^{\alpha'}}{C_n^k}\right)^{\frac{1}{k}}r^2\geq\frac{k+\alpha'}{2k}\left(\frac{m_f^{\alpha'}}{C_n^k}\right)^{\frac{1}{k}}r^2,\]
and then,
\begin{equation}\label{v esti 2}
v(r)\geq\left(\frac{k+\alpha'}{2k}\right)^{\frac{k}{k+\alpha'}}\left(\frac{m_f^{\alpha'}}{C_n^k}\right)^{\frac{1}{k+\alpha'}}r^{\frac{2k}{k+\alpha'}}=:\Lambda_1r^{\frac{2k\alpha}{k\alpha+1}},
\end{equation}
where $\Lambda_1=C(n,k,\alpha,f)$.
\par
Using \eqref{v esti 2} to estimate \eqref{v fangcheng} from above, we have
\begin{align}
(v'(r))^k&\leq\frac{nr^{k-n}}{C_n^k}\int_0^r s^{n-1}\left(\frac{f(s)}{\Lambda_1s^{\frac{2k\alpha}{k\alpha+1}}}\right)^{\alpha'}\,{\rm d}s\notag\\
&\leq\frac{nr^{k-n}}{C_n^k}\left(\frac{M_f}{\Lambda_1}\right)^{\alpha'}\int_0^r s^{n-\frac{2k}{k\alpha+1}-1}\,{\rm d}s\notag\\
&=\frac{n}{(n-\frac{2k}{k\alpha+1})C_n^k}\left(\frac{M_f}{\Lambda_1}\right)^{\alpha'}r^{\frac{k(k\alpha-1)}{k\alpha+1}},\notag
\end{align}
here we need the assumption that $\alpha>\frac{2k-n}{kn}$. The above estimate implies that
\[v'(r)\leq\left(\frac{n}{(n-\frac{2k}{k\alpha+1})C_n^k}\right)^{\frac{1}{k}}\left(\frac{M_f}{\Lambda_1}\right)^{\frac{1}{k\alpha}}r^{\frac{k\alpha-1}{k\alpha+1}}.\]
After integrating the above from $0$ to $r$, we have for any $r\geq0$,
\begin{equation}\label{v esti 3}
v(r)\leq1+\frac{k\alpha+1}{2k\alpha}\left(\frac{n}{(n-\frac{2k}{k\alpha+1})C_n^k}\right)^{\frac{1}{k}}\left(\frac{M_f}{\Lambda_1}\right)^{\frac{1}{k\alpha}}r^{\frac{2k\alpha}{k\alpha+1}}=:1+\Lambda_2r^{\frac{2k\alpha}{k\alpha+1}},
\end{equation}
where $\Lambda_2=C(n,k,\alpha,f)$. Combining \eqref{v esti 2} and \eqref{v esti 3}, we set $v(0)=1$ and then have
\[\Lambda_1r^{\frac{2k\alpha}{k\alpha+1}}\leq v(r)\leq 1+\Lambda_2 r^{\frac{2k\alpha}{k\alpha+1}},\quad r\geq0.\]
\par
By \eqref{v fangcheng} and \eqref{two4}, we have $w(t)$ is decreasing on $(-\infty,0]$ with $w(0)=(k\alpha+1)^{1/(k\alpha+1)}$ and $v(r)$ is increasing on $[0,\infty)$ with $v(0)=1$.
\par
On the one hand, we have
\begin{equation}\label{t to infty}
u(x,t)\geq w(t)v(0)=(k\alpha+1)^{\frac{1}{k\alpha+1}}\left(1+\int_t^0g(s)\,{\rm d}s\right)^{\frac{1}{k\alpha+1}}\geq (k\alpha+1)^{\frac{1}{k\alpha+1}}\left(1-m_gt\right)^{\frac{1}{k\alpha+1}},
\end{equation}
and
\begin{equation}\label{r to infty}
u(x,t)\geq w(0)v(|x|)=(k\alpha+1)^{\frac{1}{k\alpha+1}}v(|x|)\geq\Lambda_1(k\alpha+1)^{\frac{1}{k\alpha+1}}|x|^{\frac{2k\alpha}{k\alpha+1}},
\end{equation}
where the last inequality is given by Propostion \ref{pro1}. As $-t+|x|^2\to\infty$, it can be deduced by combining \eqref{t to infty} with \eqref{r to infty} that
\begin{equation}
u(x,t)\geq\left\{ \begin{aligned}
&(k\alpha+1)^{\frac{1}{k\alpha+1}}\left(1+m_gt\right)^{\frac{1}{k\alpha+1}}&\to\infty,\quad&\text{as }-t\to\infty,\\
&\Lambda_1(k\alpha+1)^{\frac{1}{k\alpha+1}}|x|^{\frac{2k\alpha}{k\alpha+1}}&\to\infty,\quad&\text{as }|x|\to\infty,
\end{aligned}\right.\notag
\end{equation}
thus $u(x,t)\to\infty$.
\par
On the other hand, as $u(x,t)\to\infty$, we prove $-t+|x|^2\to\infty$ by contradiction. We assume that there exists a sequence $\{(x^{(i)},t^{(i)})\}_{i\in\mathbb{N}}\subset\mathbb{R}^{n+1}_-$ such that $u(x^{(i)},t^{(i)})\to\infty$ as $i\to\infty$ but $-t^{(i)}+|x^{(i)}|^2\leq M$ for some constant $M>0$ and any $i\in\mathbb{N}$. By the monotonicity of $w$ and $v$, we have for any $i\in\mathbb{N}$, $u((x^{(i)},t^{(i)}))\leq w(-M)v(M^{\frac{1}{2}})$, which makes a contradiction to the assumption.
\end{proof}
To prove Theorem \ref{the2}, we further strengthen the assumptions, namely, set $f$ and $g$ are periodic functions and satisfy \eqref{f and g period}, \(\alpha>\max\{0, \frac{2k-n}{kn}\}\). Denote $u(x,t)=w(t)v(r)$ is the solution given by Theorem \ref{the1} until the end of this section. Denote $\beta\coloneqq (2k\alpha)/(k\alpha+1)$ and $p\coloneqq n-\alpha'\beta$. Then we can compute by using $\alpha>\frac{2k-n}{kn}$,
\begin{equation}\label{alpha'beta}
p=n-\frac{2k}{k\alpha+1}>0.
\end{equation}
\begin{lemma}\label{yz bound}
The following two functions defining on $[0,\infty)$:
\[y(r):=r^{-\beta}v(r), \quad  z(r):=r^{1-\beta}v'(r),\]
are both bounded on $[0,\infty)$ for some $R_0>0$. The lower and upper bounds of $y$ and $z$ are all only depending on $n,k,\alpha$ and $f$.
\end{lemma}
\begin{proof}
The lower and upper bounds of $y(r)$ follows directly from Proposition \ref{pro1}, where the bounds are depending on $n,k,\alpha$ and $f$.
\par
Substituting $v(r)=r^\beta y(r)$ and $v'(r)=r^{\beta-1} z(r)$ into \eqref{v fangcheng}, we obtain
\begin{equation}
(z(r))^k=\frac{n}{C_n^k}r^{-p} \int_0^r s^{p-1}(f(s))^{\alpha'}(y(s))^{-\alpha'}\,{\rm d}s. \label{yz fangcheng}
\end{equation}
Because of the boundedness of $f$ and $y$, \eqref{yz fangcheng} yields when $p>0$, $z(r)$ is bounded for all sufficiently large $r$.
\end{proof}
We define a positive constant as
\[\bar{f^{\alpha'}}:=\frac{1}{T_f}\int_0^{T_f}(f(s))^{\alpha'}\,{\rm d}s.\]
And we define for $r\geq R_0$, where $R_0$ is given in Lemma \ref{yz bound},
\[ J(r):=r^{-p}\int_{R_0}^r s^{p-1}(y(s))^{-\alpha'}\,{\rm d}s.\]
\begin{lemma}\label{3.2}
We have when $r\to\infty$,
\begin{equation}\label{z asym}
(z(r))^k =\frac{n\bar{f^{\alpha'}}}{C_n^k}J(r)+o(1).
\end{equation}
\end{lemma}
\begin{proof}
Firstly, we define for $r\geq0$,
\[ \tilde{f^{\alpha'}}(r):=(f(r))^{\alpha'}-\bar f^{\alpha'}.\]
By the definition of $\bar{f^{\alpha'}}$, we know that $\tilde{f^{\alpha'}}$ is also $T_f$-periodic, has zero mean, and $\tilde{F^{\alpha'}}(r):=\int_{R_0}^r \tilde{f^{\alpha'}}(s)\,{\rm d}s$ is bounded on $[R_0,\infty)$.
\par
Denote $h(r):=(y(r))^{-\alpha'}$ for $r\geq R_0$. By Lemma \ref{yz bound}, we have $|h(r)|$ is bounded on $[R_0,\infty)$. Once noticing that
\[ v'(r)=r^{\beta-1}\bigl(\beta y(r)+ry'(r)\bigr),\]
we have
\begin{equation}\label{yz huhuan} ry'(r)=z(r)-\beta y(r),\end{equation}
and hence $|ry'(r)|$ is also bounded. Therefore, $|rh'(r)|=\alpha'|y(r)|^{-\alpha'-1}|ry'(r)|$ can be also proved bounded.
By integration by parts, we have
\begin{align}
r^{-p}\int_{R_0}^r s^{p-1}\tilde{f^{\alpha'}}(s)h(s)\,{\rm d}s =&\, r^{-1}\tilde{f^{\alpha'}}(r)h(r)
-R_0^{p-1}r^{-p}\tilde{F^{\alpha'}}(R_0)h(R_0) \notag \\
& -(p-1)r^{-p}\int_{R_0}^r s^{p-2}\tilde{F^{\alpha'}}(s)h(s)\,{\rm d}s
-r^{-p}\int_{R_0}^r s^{p-2}\tilde{F^{\alpha'}}(s)(sh'(s))\,{\rm d}s.\label{right of yz estimate}
\end{align}
Using the boundedness of $\tilde{f^{\alpha'}}$, $\tilde{F^{\alpha'}}$, $|h|$, and $|rh'|$, we obtain from \eqref{right of yz estimate} that when $p>0$,
\[ r^{-p}\int_{R_0}^r s^{p-1}F(s)h(s)\,{\rm d}s=o(1),\quad r\to\infty.\]
Indeed, when $p\neq1$, the last two terms of the right-hand side of \eqref{right of yz estimate} are $O(r^{-1})$ if $p>1$
and $O(r^{-p})$ if $0<p<1$. When $p=1$, they are $O((\log r)/r)$.
\par
Since $r^{-p}\int_0^{R_0} s^{p-1}\tilde{f^{\alpha'}}(s)h(s)\,{\rm d}s$ is obviously of $O(r^{-p})$, \eqref{z asym} follows from \eqref{yz fangcheng} and the above estimate.
\end{proof}
By Lemma \ref{3.2} and \eqref{yz huhuan}, we actually have
\[ry'(r)+\beta y(r)=\kappa (J(r))^{\frac{1}{k}}+o(1),\quad r\to\infty,\]
where we denote $\kappa:=(n\bar{f^{\alpha'}}/C_n^k)^{1/k}$. And we can compute by differentiating the definition of $J(r)$ that
\[rJ'(r)=-pJ(r)+(y(r))^{-\alpha'}.\]
With $s=\ln r$, we define $y_e(s):=y(e^{s})$ and $J_e(s):=J(e^s)$. Then $y_e$ and $J_e$ satisfy the asymptotically autonomous system:
\begin{equation}\label{asym auto system}
\begin{cases}
\dot{y_e}=-\beta y_e+\kappa {J_e}^{\frac{1}{k}}+\varepsilon(s),\\
\dot{J_e}=-pJ_e+{y_e}^{-\alpha'},
\end{cases}
\end{equation}
where $\varepsilon(s)\to0$, as $s\to\infty$.
\par
We consider about the corresponding autonomous system in the below lemma previously, and prove Theorem \ref{the2} after.
\begin{lemma}\label{auto system behavior}
The autonomous system
\begin{equation}\label{auto system}
\begin{cases}
\dot{y_e}=-\beta y_e+\kappa {J_e}^{\frac{1}{k}},\\
\dot{J_e}=-pJ_e+{y_e}^{-\alpha'},
\end{cases}
\end{equation}
has a unique positive equilibrium:
\[ (y_*,J_*)=\left(C_v,\frac{C_v^{-\alpha'}}p\right),\]
where
\[C_v:=\left(\frac{\kappa^k}{p\beta^k}\right)^{\frac1{k+\alpha'}}=\left(\frac{k\alpha+1}{2k\alpha}\right)^{\frac{k\alpha}{k\alpha+1}}\left(\frac{n\bar{f^{\alpha'}}}{C_n^k(n-\frac{2k}{k\alpha+1})}\right)^{\frac{\alpha}{k\alpha+1}}.\]
Furthermore, every complete bounded trajectory of \eqref{auto system} in the positive quadrant is equal to $(y_*,J_*)$.
\end{lemma}
\begin{proof}
The functions $y_e$ and $J_e$ are both positive, which is easy to verify by their definition. For the autonomous system \eqref{auto system}, an equilibrium in the positive quadrant must satisfy
\[\beta y_e=\kappa {J_e}^{\frac{1}{k}}, \quad pJ_e={y_e}^{-\alpha'}.\]
Hence we have the unique equilibrium is $(C_v,{C_v}^{-\alpha'}/p)=(y_*,J_*)$.
\par
It remains to prove that every complete bounded trajectory of \eqref{auto system} in the positive quadrant is equal to $(y_*,J_*)$. We set in this section,
\[ A(s):=\ln\frac{y_e(s)}{y_*}, \quad B(s):=\ln\frac{J_e(s)}{J_*},\]
and define
\[ M(s):=\frac{1}{k}B(s)-A(s), \quad N(s):=-\alpha'A(s)-B(s).\]
Then the system \eqref{auto system} becomes
\begin{equation}\label{AB fangcheng}\begin{cases}
 \dot A=\beta(e^M-1),\\
 \dot B=p(e^N-1).
\end{cases}\end{equation}
And consequently,
\begin{equation}\label{MN fangcheng}\begin{cases}
 \dot M=\frac pk(e^N-1)-\beta(e^M-1),\\
 \dot N=-\alpha'\beta(e^M-1)-p(e^N-1).
\end{cases}\end{equation}
For $l(s):=e^s-s-1\geq 0$, we define the Lyapunov function of \eqref{AB fangcheng}
\[\mathcal{L}(A,B):=l(M)+\lambda l(N),\]
where $\lambda:=p/(k\alpha'\beta)$ is a positive constant.
Since $l(s)\to\infty$ as $s\to\infty$ and the linear map
$(A,B)\mapsto(M,N)$ is obviously invertible, $\mathcal L$ is a proper nonnegative function.
Moreover, we have by \eqref{MN fangcheng},
\[\dot{\mathcal{L}}=(e^M-1)\dot M+\lambda(e^N-1)\dot N=-\beta(e^M-1)^2-\lambda p(e^N-1)^2\le0,\]
where the equality holds if and only if $M=N=0$, equivalently $A=B=0$. Thus, the equilibrium $(y_*,J_*)$ is the unique invariant set on which
$\dot{\mathcal L}=0$. And we finish the proof of this lemma.
\end{proof}
\begin{proof}[Proof of Theorem \ref{the2}]
Firstly, if the solution $(y_e(s),J_e(s))$ of the asymptotically autonomous system \eqref{asym auto system} does not converge to $(y_*,J_*)$, there will exist a
$\epsilon_0>0$ and a sequence $\{s_j\}_{j=1}^{\infty}$ satisfying $s_j\to\infty$ such that for any $j$,
\[ |(y_e(s_j),J_e(s_j))-(y_*,J_*)|\geq\epsilon_0.\]
Consider the translated trajectories
\[ y_j(s):=y_e(\tau_j+s), \quad J_j(s):=J_e(\tau_j+s).\]
Because the trajectories remain in a fixed compact subset of the positive quadrant
and the right-hand side of \eqref{asym auto system} is uniformly bounded on compact time intervals,
$\{(y_j,J_j)\}$ is locally equicontinuous. And the uniformly boundedness of $\{(y_j,J_j)\}$ is trivial by definition. Therefore, we use the Arzela--Ascoli theorem again and obtain that after
passing to a subsequence, $(y_j,J_j)$ converges locally uniformly to a complete
bounded trajectory $(y_\infty,J_\infty)$ of the autonomous system \eqref{auto system}.
By Lemma \ref{auto system behavior}, this trajectory must be identically equal to $(y_*,J_*)$, which contradicts the choice of $s_j$. Hence we have
\[y_e(s)\to y_*=C_v,\quad J_e(s)\to J_*,\quad s\to\infty.\]
By the definition of $y_e(s)$ and $y(r)$, we have actually obtained that
\[\lim_{r\to\infty}\frac{v(r)}{r^\beta}=C_v.\]
Then the asymptotic behavior about $v'(r)$ can be similarly proved by using $z$. As a result, we know that
\begin{equation}\label{v zhujiejianjinxing}
\lim\limits_{r\to\infty}\frac{v(r)}{r^{\frac{2k\alpha}{k\alpha+1}}}=\lim\limits_{r\to\infty}\frac{v'(r)}{\frac{2k\alpha}{k\alpha+1}r^{\frac{k\alpha-1}{k\alpha+1}}}=C_v,
\end{equation}
and therefore, $v(r)=C_vr^{\beta}+o(r^{\beta})$ near infinity.
\par
Secondly, we denote a constant \[\bar{g}:=\frac{1}{T_g}\int_{-T_g}^0g(s)\,{\rm d}s,\]
and define a function of $t$ as \[H_3(t):=\int_t^0g(s)\,{\rm d}s+\bar{g}t.\] It is obvious that $H_3$ is bounded on $(-\infty,0)$. Begin with this and \eqref{two4}, we compute that
\[\lim\limits_{t\to-\infty}\frac{w(t)}{(-t)^{\frac{1}{k\alpha+1}}}=\lim\limits_{t\to-\infty}(k\alpha+1)^{\frac{1}{k\alpha+1}}\left(\frac{-\bar{g}t+1+H_3(t)}{-t}\right)^{\frac{1}{k\alpha+1}}=\left(\bar{g}(k\alpha+1)\right)^{\frac{1}{k\alpha+1}},\]
and as $t\to-\infty$,
\begin{align}
\frac{w(t)-\left(\bar{g}(k\alpha+1)(-t)\right)^{\frac{1}{k\alpha+1}}}{(k\alpha+1)^{\frac{1}{k\alpha+1}}}&=(1+\int_t^0g(s)\,{\rm d}s)^{\frac{1}{k\alpha+1}}-\left(-\bar{g}t\right)^{\frac{1}{k\alpha+1}}\notag\\
&=(-\bar{g}t)^{\frac{1}{k\alpha+1}}(1+\frac{1+H_3(t)}{-\bar{g}t})^{\frac{1}{k\alpha+1}}-\left(-\bar{g}t\right)^{\frac{1}{k\alpha+1}}\notag\\
&=O((-t)^{\frac{1}{k\alpha+1}-1}),
\end{align}
where to obtain the last line, we use generalized binomial theorem and omit the higher order terms. We set
$C_w=\left(\bar{g}(k\alpha+1)\right)^{\frac{1}{k\alpha+1}}$ and prove the asymptotic behavior of $w$ given in Theorem \ref{the2} when $t\to-\infty$.
\end{proof}

\section{Refined asymptotic behavior when $\alpha=1$}
In this section, we consider the case that $f$ and $g$ satisfy \eqref{f and g period}, and $\alpha=1$. In fact, if we set $n\geq3$, then for any $1\leq k\leq n$, $1=\alpha>\frac{2k-n}{kn}$. Therefore, Theorem \ref{the1} and Theorem \ref{the2} still hold. We also use $\beta$ to denote $\frac{2k\alpha}{k\alpha+1}=\frac{2k}{k+1}$ in the subsequent discussion of this section.\par
We define
\[s:=r^{\beta},\quad\Phi(s(r)):=v(r),\]
and hence by computation,
\[v'(r)=\Phi'(s)s'(r)=\beta r^{\beta-1}\Phi'(s)=\beta{s}^{\frac{\beta-1}{\beta}}\Phi'(s),\]
and
\[v''(r)=\beta^2s^{\frac{2\beta-2}{\beta}}\Phi''(s)+\beta(\beta-1)s^{\frac{\beta-2}{\beta}}\Phi'(s).\]
By \eqref{two3}, we deduce that $\Phi$ satisfies the ordinary differential equation of $\gamma\in(0,+\infty)$ that
\begin{equation}\label{Phi fangcheng}
\Phi''(s)(\Phi'(s))^{k-1}\Phi(s)+\frac{A}{s}(\Phi'(s))^k\Phi(s)=Bf(s^{\frac{1}{\beta}}),
\end{equation}
where in this section, we set 
\[A:=\frac{n\beta^{-1}-1}{k}=\frac{(n-2)k+n}{2k^2},\quad B:=\frac{n}{kC_n^k\beta^{k+1}}=\frac{n}{kC_n^k}(\frac{k+1}{2k})^{k+1}.\]
When $f=1$, \eqref{Phi fangcheng} is the same as the equation (6.1) of \cite{CJ}.
\par
We next describe the asymptotic behavior of the solution $\Phi(s)=v(r)$ at infinity by isolating the leading terms. We proceed in three steps.
\paragraph{Step 1}The RHS of \eqref{Phi fangcheng} can be decomposed into its mean value $\bar{f}:=\int_0^{T_f}f(\tau)\,{\rm d}\tau/T_f$ and its fluctuating part $\tilde{f}(r):=f(r)-\bar{f}$. We define $\Phi_0(s):=Cs$, where $C$ is a constant to be determined for balancing $B\bar{f}$. Since $\Phi_0=Cs$, we have $\Phi_0'=C$ and $\Phi_0''=0$, and we can compute that
\[\label{Phi0}\Phi_0''(s)(\Phi_0'(s))^{k-1}\Phi_0(s)+\frac{A}{s}(\Phi_0'(s))^k\Phi_0(s)=A\cdot C^{k+1}.\]
Therefore, we set $C=(B\bar{f}/A)^{1/(k+1)}$ and then have that the value of the above equation is equal to $B\bar{f}$. By a direct calculation, we can see that $C=(B\bar{f}/A)^{1/(k+1)}$ is just identical to $C_v$ with $\alpha=1$.
\paragraph{Step 2}We define $\Psi(s):=\Phi(s)-Cs$. Then, $\Psi'=\Phi'-C$ and $\Psi''=\Phi''$. By \eqref{Phi fangcheng}, we have
\begin{equation}\label{Psi fangcheng}
\Psi''(s)(\Psi'(s)+C)^{k-1}(\Psi(s)+Cs)+\frac{A}{s}(\Psi'(s)+C)^k(\Psi(s)+Cs)=Bf(s^{\frac{1}{\beta}}).
\end{equation}
Noticing identities
\[(\Psi'(s)+C)^{k-1}=C^{k-1}+\sum_{i=1}^{k-1}C_{k-1}^iC^{k-1-i}(\Psi'(s))^i,\]
and
\[(\Psi'(s)+C)^{k}=C^{k}+kC^{k-1}\Psi'(s)+\sum_{i=2}^{k}C_{k}^iC^{k-i}(\Psi'(s))^i,\]
we rewrite \eqref{Psi fangcheng} into
\begin{equation}\label{N fangcheng}
C^k\left(s\Psi''(s)+kA\Psi'(s)+A\frac{\Psi(s)}{s}\right)+\mathcal{N}\Psi=B\tilde{f}(s^{\frac{1}{\beta}}),
\end{equation}
where the functional $\mathcal{N}$ is defined as
\begin{align}\mathcal{N}\Psi\coloneqq & C^{k-1}\Psi(s)\Psi''(s)+kAC^{k-1}\frac{\Psi(s)}{s}\Psi'(s)\notag\\
&+\sum_{i=1}^{k-1}C_{k-1}^iC^{k-1-i}(\Psi(s)+Cs)(\Psi'(s))^i\Psi''(s)+\sum_{i=2}^{k}C_{k}^iC^{k-i}(A\frac{\Psi(s)}{s}+AC)(\Psi'(s))^i.\notag
\end{align}
And we define two functionals
\[\mathcal{Y}(y,z)\coloneqq(Cs+y)(C+z)^{k-1}-C^ks,\]
\[\mathcal{Z}(y,z)\coloneqq\frac{A}{s}\left((Cs+y)(C+z)^k-C^ky-kC^ksz-C^{k+1}s\right),\]
where $y$ and $z$ are functions of $s$. It is easy to verify that
\begin{equation}\mathcal{N}\Psi=\mathcal{Y}(\Psi,\Psi')\Psi''+\mathcal{Z}(\Psi,\Psi').\label{NP biaodashi}\end{equation}
\par
We claim that there exists a $T_f$-periodic function $F_1$ of $r$ satisfying
\[F_1''(r)=\frac{4k^2B}{(k+1)^2C^k}\tilde{f}(r),\]
and
\[\int_0^{T_f} F_1'(\tau)\,{\rm d}\tau=0,\quad\int_0^{T_f} F_1(\tau)\,{\rm d}\tau=0.\]
In fact, the proof of this claim is trivial: one only needs to observe that $\tilde{f}$ has zero mean over a period, and then a direct integration yields the desired conclusion. It is similarly easy to prove that there also exists a $T_f$-periodic function $F_2$ of $r$ satisfying
\[F_2''(r)=-(n-3)F_1'(r),\quad\text{and}\quad\int_0^{T_f}F_2(\tau)\,{\rm d}\tau=0.\]
\par
We difine the following two functions:
\[P_1(s)\coloneqq s^{-\frac{1}{k}}F_1(s^{\frac{1}{\beta}}),\quad P_2(s)\coloneqq s^{-\frac{k+3}{2k}}F_2(s^{\frac{1}{\beta}}),\]
and define a linear operator by
\[\mathcal{L}P\coloneqq sP''(s)+kAP'(s)+A\frac{P(s)}{s}.\]
Then by direct computation, we have
\begin{align}\mathcal{L}P_1&=\frac{(k+1)^2}{4k^2}F_1''(s^{\frac{1}{\beta}})+\frac{(n-3)(k+1)^2}{4k^2}\frac{F_1'(s^{\frac{1}{\beta}})}{s^{\frac{1}{\beta}}}+\frac{k+1}{k^2}\frac{F_1(s^{\frac{1}{\beta}})}{s^{\frac{2}{\beta}}}\notag\\
&=\frac{B}{C^k}\tilde{f}(s^{\frac{1}{\beta}})+\frac{(n-3)(k+1)^2}{4k^2}\frac{F_1'(s^{\frac{1}{\beta}})}{s^{\frac{1}{\beta}}}+\frac{k+1}{k^2}\frac{F_1(s^{\frac{1}{\beta}})}{s^{\frac{2}{\beta}}},\label{LP1}
\end{align}
and
\begin{align}
\mathcal{L}P_2&=\frac{(k+1)^2}{4k^2}\frac{F_2''(s^{\frac{1}{\beta}})}{s^{\frac{1}{\beta}}}+\frac{(n-5)(k+1)^2}{4k^2}\frac{F_2'(s^{\frac{1}{\beta}})}{s^{\frac{2}{\beta}}}-\frac{(k+1)(kn-5k+n-9)}{4k^2}\frac{F_2(s^{\frac{1}{\beta}})}{s^{\frac{3}{\beta}}}\notag\\
&=-\frac{(n-3)(k+1)^2}{4k^2}\frac{F_1'(s^{\frac{1}{\beta}})}{s^{\frac{1}{\beta}}}+\frac{(n-5)(k+1)^2}{4k^2}\frac{F_2'(s^{\frac{1}{\beta}})}{s^{\frac{2}{\beta}}}-\frac{(k+1)(kn-5k+n-9)}{4k^2}\frac{F_2(s^{\frac{1}{\beta}})}{s^{\frac{3}{\beta}}}.\label{LP2}
\end{align}
\paragraph{Step 3}We define $Q(s):=\Psi(s)-P_1(s)-P_2(s)$, so that $\Phi(s)=Cs+P_1(s)+P_2(s)+Q(s)$. We will consider $Cs+F_1(s)+F_2(s)$ as the leading term and $Q(s)$ as the remainder term. By \eqref{N fangcheng}, we have
\[C^k(\mathcal{L}P_1+\mathcal{L}P_2+\mathcal{L}Q)+\mathcal{N}(P_1+P_2+Q)=B\tilde{f}(s^{\frac{1}{\beta}}).\]
Substituting \eqref{LP1} and \eqref{LP2} into the above equation, we obtain
\begin{equation}C^k\mathcal{L}Q+\mathcal{N}(P_1+P_2+Q)=-C^kQ_0(s),\label{LQ equ1}\end{equation}
where we define
\[Q_0(s)\coloneqq\frac{k+1}{k^2}\frac{F_1(s^{\frac{1}{\beta}})}{s^{\frac{2}{\beta}}}+\frac{(n-5)(k+1)^2}{4k^2}\frac{F_2'(s^{\frac{1}{\beta}})}{s^{\frac{2}{\beta}}}-\frac{(k+1)(kn-5k+n-9)}{4k^2}\frac{F_2(s^{\frac{1}{\beta}})}{s^{\frac{3}{\beta}}}.\]
And by \eqref{NP biaodashi}, we have
\[\mathcal{N}(P+Q)=\mathcal{Y}(P+Q,P'+Q')(P''+Q'')+\mathcal{Z}(P+Q,P'+Q'),\]
where we set $P= P_1+P_2$ henceforth, and then
\begin{align}\mathcal{N}(P+Q)=&\mathcal{N}P+\mathcal{Y}(P+Q,P'+Q')Q''+\left(\mathcal{Y}(P+Q,P'+Q')-\mathcal{Y}(P,P')\right)P''\notag\\&+\mathcal{Z}(P+Q,P'+Q')-\mathcal{Z}(P,P').\notag
\end{align}
\par
Noticing $\mathcal{Y}$ and $\mathcal{Z}$ are both $C^1$, we have by the fundamental theorem of calculus that
\begin{equation}\mathcal{N}(P+Q)=\mathcal{N}P+\Theta_0(s)Q+\Theta_1(s)Q'+\Theta_2(s)Q'',\label{N(P+Q)}\end{equation}
where we denote
\[\Theta_0(s)\coloneqq\int_0^1\left(\mathcal{Y}_y(P+\theta Q,P'+\theta Q')P''+\mathcal{Z}_y(P+\theta Q,P'+\theta Q')\right)\,{\rm d}\theta,\]
\[\Theta_1(s)\coloneqq\int_0^1\left(\mathcal{Y}_z(P+\theta Q,P'+\theta Q')P''+\mathcal{Z}_z(P+\theta Q,P'+\theta Q')\right)\,{\rm d}\theta,\]
and
\[\Theta_2(s)\coloneqq \mathcal{Y}(P+Q,P'+Q').\]
Combining \eqref{LQ equ1} and \eqref{N(P+Q)}, we have
\[C^k\mathcal{L}Q+\mathcal{N}P+\Theta_0(s)Q+\Theta_1(s)Q'+\Theta_2(s)Q''=-C^kQ_0(s).\]
By the definition of $\mathcal{L}$, we rewrite the above equation into
\[\left(C^ks+\Theta_2(s)\right)Q''(s)+\left(kAC^k+\Theta_1(s)\right)Q'(s)+\left(\frac{AC^k}{s}+\Theta_0(s)\right)Q(s)=-\mathcal{N}P(s)-C^kQ_0(s).\]
For normalization, we divide both sides of the above equation by $(C^ks+\Theta_2(s))/s^2$, which yields
\[s^2Q''(s)+\frac{kAC^ks+s\Theta_1(s)}{C^ks+\Theta_2(s)}sQ'(s)+\frac{AC^ks+s^2\Theta_0(s)}{C^ks+\Theta_2(s)}Q(s)=-s^2\frac{\mathcal{N}P(s)+C^kQ_0(s)}{C^ks+\Theta_2(s)}.\]
We denote
\[R_1(s)\coloneqq \frac{kAC^ks+s\Theta_1(s)}{C^ks+\Theta_2(s)}-kA,\quad R_2(s)\coloneqq\frac{AC^ks+s^2\Theta_0(s)}{C^ks+\Theta_2(s)}-A,\]
and
\[R(s)\coloneqq-s^2\frac{\mathcal{N}P+C^kQ_0(s)}{C^ks+\Theta_2(s)},\]
then we finally obtain the ODE of $Q$ as
\begin{equation}\label{Q ODE}s^2Q''(s)+(kA+R_1(s))sQ'(s)+(A+R_2(s))Q(s)=R(s).\end{equation}
\begin{lemma}\label{R1R2 o(1)}
As $s\to\infty$, we have
\[R_1(s)=o(1),\quad R_2(s)=o(1),\quad\text{and}\quad R(s)=O(s^{-\frac{1}{k}}).\]
\end{lemma}
\begin{proof}
By \eqref{v zhujiejianjinxing}, we have that as $s\to\infty$,
\[\frac{\Psi(s)}{s}=\frac{v(r)}{r^{\beta}}-C\to 0,\quad \Psi'(s)=\beta r^{\beta-1}v'(r)-C\to 0,\]
and therefore
\[\frac{Q(s)}{s}=\frac{\Psi(s)-P(s)}{s}=\frac{\Psi(s)}{s}-s^{-\frac{k+1}{k}}F_1(s^{\frac{1}{\beta}})-s^{-\frac{3(k+1)}{2k}}F_2(s^{\frac{1}{\beta}})\to0,\]
where we use the boundedness of $F_1$ and $F_2$. And it is similarly easy to verify that $Q'(s)\to 0$ as $s\to\infty$. In fact, by the definition of $P_1$ and $P_2$, we further have $P_1'(s)=O(s^{-1-1/k})$, $sP_1''(s)=O(s^{-1-1/k})$, $P_2'(s)=O(s^{-3(k+1)/(2k)})$ and $sP_2''(s)=O(s^{-3(k+1)/(2k)})$ as $s\to\infty$, which are all infinitesimal.
\par
Firstly, it is straightforward to see
\[\frac{\Theta_2(s)}{C^ks}=\frac{(Cs+P+Q)(C+P'+Q')^{k-1}-C^ks}{C^ks}\to 0,\quad s\to\infty.\]
Secondly, by the definition of $\mathcal{Y}$ and $\mathcal{Z}$, we have
\[\mathcal{Y}_y(y,z)=(C+z)^{k-1},\quad \mathcal{Z}_y(y,z)=\frac{A}{s}((C+z)^k-C^k),\]
and therefore,
\[s\Theta_0(s)=sP''(s)\int_0^1 (C+P'+\theta Q')\,{\rm d}\theta+A\int_0^1((C+P'+\theta Q')^k-C^k)\,{\rm d}\theta\to 0,\quad\text{as }s\to\infty.\]
Thirdly, by computing $\mathcal{Y}_z$ and $\mathcal{Z}_z$, we can similarly obtain that $\Theta_1(s)\to0$ as $s\to\infty$. Finally, we have
\[R_1(s)=\frac{kA+\frac{\Theta_1(s)}{C^k}}{1+\frac{\Theta_2(s)}{C^ks}}-kA=o(1),\quad\text{as }s\to\infty,\]
and
\[R_2(s)=\frac{A+\frac{s\Theta_0(s)}{C^k}}{1+\frac{\Theta_2(s)}{C^ks}}-A=o(1),\quad\text{as }s\to\infty.\]
\par
As for the estimate of $R(s)$, we omit the non-essential details and focus only on estimating the leading term that affects the asymptotic behavior as $s\to\infty$. By the definition of $Q_0(s)$ and $\mathcal{N}P$, we can see that $sQ_0(s)=O(s^{-1/k})$ and $s\mathcal{N}P=O(s^{-1-2/k})$ as $s\to\infty$. Therefore, we have as $s\to\infty$,
\[R(s)=-\frac{s\mathcal{N}P+C^ksQ_0(s)}{1+\frac{\Theta_2(s)}{C^ks}}=O(s^{-\frac{1}{k}}).\]
\end{proof}
We denote $\tilde{s}\coloneqq \ln s$ and define $\tilde{Q}(\tilde{s})\coloneqq Q(e^{\tilde{s}})$, $\tilde{R}(\tilde{s})\coloneqq R(e^{\tilde{s}})$, $\tilde{R}_1(\tilde{s})\coloneqq R_1(e^{\tilde{s}})$ and $\tilde{R}_2(\tilde{s})\coloneqq R_2(e^{\tilde{s}})$, then
\[\tilde{Q}'=sQ',\quad\tilde{Q}''=sQ'+s^2Q''.\]
By \eqref{Q ODE}, we have that $\tilde{Q}$ satisfies the following ODE:
\begin{equation}\tilde{Q}''(\tilde{s})+(kA-1+\tilde{R}_1(\tilde{s}))\tilde{Q}'(\tilde{s})+(A+\tilde{R}_2(\tilde{s}))\tilde{Q}(\tilde{s})=\tilde{R}(\tilde{s}).\label{tildeQ fangcheng}\end{equation}
\par
To determine the exact asymptotic order of $\tilde{Q}$ at infinity, we start with the constant-coefficient equation in the limiting case:
\[U''(\tau)+(kA-1)U'(\tau)+AU(\tau)=0.\]
The limiting
characteristic polynomial of \eqref{tildeQ fangcheng} is
\[p(\lambda)=\lambda^2+\frac{(n-4)k+n}{2k}\lambda+\frac{(n-2)k+n}{2k^2},\]
and the two roots are
\[\lambda_{\pm}=-\frac{(n-4)k + n}{4k} \pm \frac{1}{4k} \sqrt{((n-4)k + n)^2 - 8((n-2)k + n)},\]
which agrees with that given in Section 6.1 of \cite{CJ}. It is worth noting that the real parts of $\lambda_{\pm}$ are both negative, and we denote the larger as $\lambda_*$.
\par
Following the steps given in Secition 4 of \cite{ABL}, we denote $W=\begin{pmatrix} U & U' \end{pmatrix}^{\top}$, then
\[W'=(A_{\infty}+E(\tau))W+H(\tau),\]
where we denote
\[A_{\infty}\coloneqq\begin{pmatrix}0 & 1\\ -A & -(kA-1) \end{pmatrix},\quad E(\tau)\coloneqq\begin{pmatrix}0 & 0\\ -\tilde{R}_1(\tau) & -\tilde{R}_2(\tau) \end{pmatrix},\quad H(\tau)\coloneqq\begin{pmatrix}0\\ -\tilde{R}(\tau) \end{pmatrix}.\]
By Lemma \ref{R1R2 o(1)}, we know that for any $\epsilon>0$, there exists a sufficiently large $\tau_0$ such that for any $\tau\geq\tau_0$, $||E(\tau)||\leq\epsilon$. Then, by the asymptotic stability result, we have 
\[||W(\tau)||\leq C_{\epsilon}\left(e^{(\lambda_*+\epsilon)\tau}+\int_{\tau_0}^{\tau}e^{(\lambda_*+\epsilon)(t-\theta)}e^{-\frac{\theta}{k}}\,{\rm d}\theta\right),\]
where $C_{\epsilon}$ is a constant depending on $\epsilon$. After a case analysis, we can conclude that for $\mu\coloneqq\max\{\lambda_*,-1/k\}$,
\[U(\tau)=O(e^{(\mu+\epsilon)\tau}),\quad\forall \epsilon>0.\]
Thus by selecting $\epsilon<-\mu$, we have $U(\tau)\to0$ as $\tau\to\infty$.
\par
Applying the estimate just derived, we turn back to investigate the solution of \eqref{Q ODE}. As $s\to\infty$, we can estimate $Q(s)=O(s^{\mu+\epsilon})$, and then by the definition of $R_i$ for $i=1, 2$,
\[R_i(s)=O(s^{-\frac{k+1}{k}})+O(s^{\mu+\epsilon-1}),\]
thus $R_i$ belongs to $L^1(0,\infty)$. Again following \cite{ABL}, we find
\[Q(s)=O(s^{\mu}),\quad\text{as }s\to\infty,\]
which yields directly that
\[\Phi(s)=Cs+P_1(s)+P_2(s)+O(s^{\mu}),\] 
and then
\[v(r)=Cr^{\frac{2k}{k+1}}+r^{-\frac{2}{k+1}}F_1(r)+r^{-\frac{k+3}{k+1}}F_2(r)+O(r^{\frac{2k\mu}{k+1}}).\]
To emphasize the decay order, we denote
\[M_{n,k}\coloneqq \min\{-\frac{2k\lambda_*}{k+1}, \frac{2}{k+1}\},\]
and then we finish the proof of Corollary \ref{the3}.

\newpage

\newpage
\bibliographystyle{amsplain}

\end{document}